\documentclass[a4paper,11pt, reqno]{amsart}
\usepackage{amssymb,amsthm,amsmath}
 \usepackage{enumerate}
\usepackage{ifthen}
\usepackage{graphicx}
\usepackage{cite}
\usepackage{tikz} 
\usetikzlibrary{arrows.meta}
\usepackage{amsmath,amscd,amssymb}
\usepackage{latexsym}
\usepackage[colorlinks,citecolor=blue,pagebackref,hypertexnames=false]{hyperref}

\numberwithin{equation}{section}
\allowdisplaybreaks[2]
\theoremstyle{plain}
\newtheorem{theorem}{Theorem}[section]

\newtheorem{lemma}[theorem]{Lemma}

\newtheorem{proposition}[theorem]{Proposition}

\theoremstyle{definition}

\theoremstyle{remark}
\newtheorem{remark}[theorem]{Remark}

\newtheorem{notation}[theorem]{Notation}

\newtheorem{case[theorem]}{Case}

\def\norm#1.#2.{\lVert#1\rVert_{#2}}

\title[Orthonormal Strichartz estimates for Dunkl-Schr\"{o}dinger equation with Sobolev regularity]{Orthonormal Strichartz estimates for Dunkl-Schr\"{o}dinger equation of initial data with Sobolev regularity
}

\author{Guoxia Feng}
\author{Shyam Swarup Mondal} 
\author{Manli Song}
\thanks{Corresponding author: Manli Song (mlsong@nwpu.edu.cn)}
\author{Huoxiong Wu}

\address{\endgraf School of Mathematical Sciences, Xiamen University, Xiamen 361005, China}
\email{gxfeng@mail.nwpu.edu.cn}

\address{\endgraf  Stat-Math Unit, Indian Statistical Institute  Kolkata, BT Road,  Baranagar, Kolkata  700108, India }
\email{mondalshyam055@gmail.com}

\address{\endgraf School of Mathematics and Statistics, Northwestern Polytechnical University, Xi'an, Shaanxi 710129, China}
\email{mlsong@nwpu.edu.cn}

\address{\endgraf School of Mathematical Sciences, Xiamen University, Xiamen 361005, China}
\email{huoxwu@xmu.edu.cn}


\date{\today}
\begin{document}
	
	\maketitle

	\allowdisplaybreaks

\begin{abstract}
Let $\Delta_\kappa$ be the Dunkl-Laplacian on $\mathbb{R}^n$. The main aim of this paper is to investigate the orthonormal Strichartz estimates for the Schrödinger equation with initial data from the homogeneous Dunkl-Sobolev space $\dot{H}_\kappa^s (\mathbb{R}^n)$. Our approach is based on restricted weak-type orthonormal estimates, frequency-localized estimates for the Dunkl-Schr\"odinger propagator $e^{it\Delta_\kappa}$, and a series of successive real and complex interpolation techniques. 
\end{abstract}
	\tableofcontents

	\section{Introduction}
The classical (single-function) Strichartz estimates for the free Schr\"{o}dinger propagator $e^{it\Delta}$ of initial data $f\in L^2( \mathbb{R}^n)$ on the Euclidean space $\mathbb{R}^n$, are stated as follows:
\begin{equation}\label{Lebesgue}
\|e^{it\Delta}f\|_{L^{2q}(\mathbb{R},L^{2p}(\mathbb{R}^n))}\lesssim \|f\|_{L^2(\mathbb{R}^n)},
\end{equation}
for all spatial dimensions $n\geq1$ and Schr\"{o}dinger admissible pairs $(p,q)$ satisfying $p,q\geq1$, $\frac{2}{q}+\frac{n}{p}=n$ and $(q,p,n)\neq (1,\infty,2)$, see \cite{Str}. The proof of these estimates is a combination of an abstract functional analysis argument known as the $TT^*$ duality argument and an $L^1\rightarrow L^\infty$ dispersive estimate, except the endpoint case $(q,p)=(1,\frac{n}{n-2})$ for $n\geq3$, which was proved by Keel-Tao \cite{KT}.

By a scaling argument and the Sobolev embedding theorem, it is well known that  under the  following  necessary condition
\begin{equation*}
\frac{2}{q}+\frac{n}{p}=n-2s\text{ and }s\in [0,\frac{n}{2}),
\end{equation*}
the estimate  \eqref{Lebesgue} reduces into    the  following  homogeneous Sobolev Strichartz estimate  
\begin{equation*} \|e^{it\Delta}f\|_{L^{2q}(\mathbb{R},L^{2p}(\mathbb{R}^n))}\lesssim \|f\|_{\dot{H}^s(\mathbb{R}^n)}.
\end{equation*} 
In fact, an alternative approach to obtain such estimates in the homogeneous Sobolev space is based on the Littlewood–Paley theory. More precisely, first  one can establish the estimates of the form \eqref{Lebesgue} for initial data $f\in L^2(\mathbb{R}^d)$
that are frequency localized to annuli and then upgrade to general data by putting together the estimates for each dyadic piece by the Minkowski inequality and the Littlewood–Paley square function theorem.

	In recent years, a significant attention has been devoted by numerous researchers to investigate  single-function Strichartz estimates   to versions involving systems of orthonormal functions taking the form
    \begin{equation}\label{abstract}
		\left\|\sum_{j=1}^\infty\lambda_j|e^{itL}f_j|^2\right\|_{L^q(\mathbb{R},L^p(\mathbb{R}^n))} \lesssim \|\{\lambda_j\}_{j=1}^\infty\|_{\ell^{\alpha}},
	\end{equation}
	for families of orthonormal data  $\{f_j\}_{j=1}^\infty$ in a given Hilbert space, such as the Lebesgue space $L^2(\mathbb{R}^n)$ or the homogeneous Sobolev space $\dot{H}^s(\mathbb{R}^n)$. Here  $\{\lambda_j\}_{j=1}^\infty \in \ell^{\alpha}$  and $ q, p, \alpha \in [1, \infty]$ satisfying some particular
	conditions. We now briefly summarize some particular known results of the form \eqref{abstract}.
\begin{itemize}
    \item Let $L=\Delta$ be the classical Laplacian. For initial data in the Lebesgue space $L^2(\mathbb{R}^n)$, the orthonormal Strichartz estimates \eqref{abstract} was pioneered by Frank–Lewin–Lieb–Seiringer \cite{FLLS}, and later extended by Frank-Sabin \cite{FS} for more general $p, q, \alpha$ with applications to the well-posedness of the Hartree equation in Schatten spaces.   They also  generalized the theorems of Stein-Tomas and Strichartz about surface restrictions of Fourier transforms to systems of orthonormal functions with an optimal dependence on the number of functions.
We point out that Frank-Sabin \cite{FS} established a crucial duality principle, which could rephrase Strichartz estimates for orthonormal families of initial data in terms of certain Schatten space
estimates.

     \item Let $L=\Delta$ be the classical Laplacian. For the initial data in the homogeneous Sobolev space $\dot{H}^s(\mathbb{R}^n)$ with the necessary scaling condition $\frac{2}{q}+\frac{n}{p}=n-2s$, it was established by Bez-Hong-Lee-Nakamura-Sawano \cite{Bez-Hong-Lee-Nakamura-Sawano}. We shall highlight that in the case of a single function, one can proceed by first proving the desired estimates for initial data in the Lebesgue space $L^2(\mathbb{R}^n)$
with localized frequency on dyadic annuli and upgrade to general data in the homogeneous Sobolev space $\dot{H}^s(\mathbb{R}^n)$ via Littlewood–Paley
theory. In contrast with the classical single-function Strichartz estimates, it seems difficult to upgrade the frequency global estimates from the frequency localized estimates in the case of generalised Strichartz estimates for orthonormal systems. Although they can easily obtain estimates of the form \eqref{abstract} with the frequency
localisation. The argument is based
on upgrading the frequency localized estimates to general data and is carried out via a
succession of interpolation arguments.

\item Let $L=\Delta_\kappa$ be the Dunkl Laplacian on $\mathbb{R}^n$.    For initial data in weighted $L^2(\mathbb{R}^n)$ space, the orthonormal Strichartz estimates of the form \eqref{abstract} were proved by Kumar-Pradeep-Mondal-Mejjaoli in \cite{shyam}  by considering  Strichartz's restriction theorem for the Fourier-Dunkl transform for certain surfaces, namely, paraboloid, sphere, and hyperboloid, and its generalization to the family of orthonormal functions. Also,    inspired by the duality principle in Frank-Sabin \cite{FS}, Mondal-Song \cite{MS} obtained orthonormal Strichartz estimates associated with  $(k, a)$-generalized Laguerre operator and Dunkl operator with the initial data in the $L^2$-weighted space $L^2_\kappa(\mathbb{R}^n)$.

    \item  Let $L$ be a non-negative, self-adjoint operator on $L^2(X,d\mu)$, where $(X,\mu)$ is a measure space. Recently, under the assumption that the kernel of the Schr\"{o}dinger propagator $e^{itL}$ satisfies a uniform $L^\infty$-decay estimate, Feng-Mondal-Song-Wu \cite{FMSW} establish the orthonormal Strichartz estimates for the Schr\"{o}dinger propagator $e^{itL}$ on $L^2(X,\mu)$, which extends the work of Keel-Tao \cite{KT} for a single function. 
\end{itemize}
For additional outcomes analogous to \eqref{abstract}, we refer the reader to the 
 recent papers, including \cite{Feng-Song, MS, Nguyen, nakamura, BEZ, Hoshiya, Ren-Zhang,Wang-Zhang-Zhang, FS1, GMS, shyam} and references therein.

  The orthonormal Strichartz estimates of the form \eqref{abstract} are important, particularly from two different perspectives. First, it represents a refinement of the classical Strichartz estimates for a single function, which have been utilized in the analysis of nonlinear Schrödinger equations for the last five decades.  Noting that by taking 
$f_j=0$ for all  $j\geq 2$  in the orthonormal Strichartz estimates \eqref{abstract}, we recover the classical single-function version.  Second, the estimates play a crucial role in the analysis of the Hartree equations, which describe systems of infinitely many fermions. This was first investigated in \cite{LS}, where local and global well-posedness was established under various conditions.  Note that the idea of generalizing classical inequalities from a single-function input to an orthonormal family is not a new topic, and the first pioneering work of this kind of generalization can be traced back to the work of  Lieb-Thirring  \cite{lieb, liebb} and this Lieb-Thirring inequality generalizes extended versions of certain Gagliardo–Nirenberg–Sobolev inequalities, which is considered to be one of the most important tools for the study of the stability of matter \cite{lieb}.  

Motivated by the recent works of Bez-Hong-Lee-Nakamura-Sawano in  \cite{Bez-Hong-Lee-Nakamura-Sawano},  the main aim of this paper is to investigate the orthonormal Strichartz estimates for the Schr\"odinger equation associated with the  Dunkl-Laplacian $\Delta_\kappa$  on $\mathbb{R}^n$   with initial data from the homogeneous Sobolev space $\dot{H}^s(\mathbb{R}^n)$.  Note that the 	Dunkl operators are differential-difference operators that generalize the standard partial derivatives.   These operators were introduced by Charles Dunkl \cite{dun1991} using a finite reflection group and a parameter function on its root system.   Applications of Dunkl operators in mathematics and mathematical physics are numerous. For example, they play an important role in the integration of quantum many-body problems of the Calogero-Moser-Sutherland type and are employed in the study of probabilistic processes. Additionally, these operators have had a lasting impression on the study of special functions in one and multiple variables as well as orthogonal polynomials,  see \cite{Dunkl book}.  

Let $\Delta_\kappa$ be the Dunkl Laplacian. We refer the reader 
 to  Section \ref{sec2} for detailed definitions of  $N=2\gamma+n$, $L_\kappa^p(\mathbb{R}^n)$ and $\dot{H}_\kappa^s(\mathbb{R}^n)$. The single-function Strichartz estimate for the Dunkl-Schr\"{o}dinger propagator $e^{it\Delta_\kappa}$ has been established by Mejjaoli  \cite{Mejjaoli2009} in the following form.
\begin{theorem}\label{Strichartz-single}
Assume $p,q\geq1$, $(q,p,N)\neq (1,\infty,2)$ and $\frac{2}{q}+\frac{N}{p}=N$. Then the Dunkl-Schr\"{o}dinger operator $e^{it\Delta_\kappa}$ is bounded from $L_\kappa^2(\mathbb{R}^n)$ to $L^{2q}(\mathbb{R}, L_\kappa^{2p}(\mathbb{R}^n))$, i.e.,
\begin{equation}\label{Single-Dunkl}
\|e^{it\Delta_\kappa}f\|_{L^{2q}(\mathbb{R},L_\kappa^{2p}(\mathbb{R}^n))}\lesssim \|f\|_{L_\kappa^2(\mathbb{R}^n)}.
\end{equation}
\end{theorem}
By the Dunkl-Sobolev embedding relation \eqref{Sobolev-Embedding}, we immediately obtain estimates of the form \eqref{Single-Dunkl} in the homogeneous Dunkl-Sobolev space
\begin{equation}\label{Sobolev-single}
\|e^{it\Delta_\kappa}f\|_{L^{2q}(\mathbb{R},L_\kappa^{2p}(\mathbb{R}^n))}\lesssim \|f\|_{\dot{H}_\kappa^s(\mathbb{R}^n)},
\end{equation}
where $1\leq p<\infty,1\leq q\leq \infty$, $s\in [0,\frac{N}{2})$ and $\frac{2}{q}+\frac{N}{p}=N-2s$.

Recently, Mondal-Song \cite{MS} obtained a substantial generalization of \eqref{Single-Dunkl} for system of orthonormal
functions as follows. 
\begin{theorem} \label{M-S}
If $p,q,n\geq 1$ satisfy 
$1\leq p<\frac{N+1}{N-1}$, $\frac{2}{q}+\frac{N}{p}=N$ and $\alpha=\frac{2p}{p+1}$, then we have
\begin{equation*}
			\bigg\|\sum_{j=1}^\infty \lambda_j|e^{it\Delta_\kappa}f_j|^2\bigg\|_{L^q(\mathbb{R}, L_\kappa^p(\mathbb{R}^n))}\lesssim \|\{\lambda_j\}^\infty_{j=1}\|_{\ell^{\alpha}},
\end{equation*}	
  holds for all families of orthonormal functions $\{f_j\}_{j=1}^\infty$ in $L_\kappa^2(\mathbb{R}^n)$ and all sequences $\{\lambda_j\}^\infty_{j=1}$ in $\ell^{\alpha}$. 
\end{theorem}
Our goal  is to establish orthonormal Strichartz estimates for the Dunkl-Schr\"odinger propagator  $e^{it\Delta_\kappa}$ in the homogeneous Dunkl-Sobolev spaces framework. In order to state our results precisely, first we establish some notation.
\begin{notation}
We introduce the following points:
\begin{equation*}
\begin{array}{llllll}
&O = (0,0),\quad & A = (\frac{N-1}{N+1},\frac{N}{N+1}), \quad & B = (1,0), \quad & C = (0,1), \\
& D = (\frac{N-2}{N},1), & E = (\frac{N-1}{2N},\frac{1}{2}), \quad & F  = (\frac{N}{N+2},\frac{N}{N+2}),\quad & G= (\frac{N}{2},0).
\end{array}
\end{equation*}  

For $N\geq 2$, see Figure 1. We point out that for $N=2$, the points $C$ and $D$ coincide. 
\begin{center}\begin{tikzpicture} \draw[thick,->, >=latex] (0,0) -- (0,7) node[left] {\(\frac{1}{q}\)};\draw[thick,->, >=latex] (-5/2,0) -- (7,0) node[below] {\(\frac{1}{p}\)};\node at (0,6) [above right]{\(1\)};\node at (6,0) [below]{\(1\)};\node at (0,0) [below]{\(O\)};\node at (0,6) [above left]{\(C\)};\node at (2,6) [above]{\(D\)};\node at (6,0) [above right]{\(B\)};\node at (18/5,18/5) [above right]{\(F\)}; \node at (3,9/2) [above right]{\(A\)};\node at (2,3) [above left]{\(E\)};\node at (0,3) [left]{\(\frac{1}{2}\)};\node at (2,0) [below]{\(\frac{N-1}{2N}\)};\node at (0,9/2) [left]{\(\frac{N}{N+1}\)};\node at (3,0) [below]{\(\frac{N-1}{N+1}\)};\node at (0,18/5) [left]{\(\frac{N}{N+2}\)};\node at (18/5,0) [below]{\(\frac{N}{N+2}\)};\node at (3,-1) [below] {\(\text{Figure 1: }N\geq2\)};\draw[thick] (6,0) -- (6,6);\draw[thick] (6,0) -- (2,6);\draw[thick] (0,6) -- (6,6);\draw[thick] (0,0) -- (3,9/2);\draw[dotted] (0,3)--(2,3);\draw[dotted] (2,0)--(2,3);\draw[dotted] (3,0)--(3,9/2);\draw[dotted] (0,9/2)--(3,9/2);\draw[dotted] (18/5,0)--(18/5,18/5);\draw[dotted] (0,18/5)--(18/5,18/5);\end{tikzpicture}\end{center}


For $1\leq N<2$, see Figure 2. We point out that for $N=1$, the points $A$, $E$ and $G$ coincide. 
\begin{center}
\begin{tikzpicture}
            \draw[thick,->, >=latex] (0,0) -- (0,7) node[left] {\(\frac{1}{q}\)};
            \draw[thick,->, >=latex] (-5/2,0) -- (7,0) node[below] {\(\frac{1}{p}\)};
            \node at (0,6) [above right]{\(1\)};
            \node at (6,0) [below]{\(1\)};
            \node at (0,4.5) [right]{\(G\)};
            \node at (0,0) [below]{\(O\)};
            \node at (0,6) [above left]{\(C\)};
            \node at (-2,6) [above]{\(D\)};
            \node at (6,0) [above right]{\(B\)};
            \node at (18/7,18/7) [above right]{\(F\)};
            \node at (6/5,18/5) [above right]{\(A\)};
            \node at (1,3) [above left]{\(E\)};
            \node at (0,3) [left]{\(\frac{1}{2}\)};
            \node at (1,0) [below]{\(\frac{N-1}{2N}\)};
            \node at (0,18/5) [left]{\(\frac{N}{N+1}\)};
            \node at (6/5,0) [below right]{\(\frac{N-1}{N+1}\)};
            \node at (0,18/7) [left]{\(\frac{N}{N+2}\)};
            \node at (18/7,0) [below]{\(\frac{N}{N+2}\)};
            \node at (3,-1) [below] {\(\text{Figure 2: }1\leq N<2\)};
            \draw[thick] (6,0) -- (6,6);
            \draw[thick] (6,0) -- (-2,6);
            \draw[thick] (-2,6) -- (6,6);
            \draw[thick] (0,0) -- (6/5,18/5);
            \draw[dotted] (0,3)--(1,3);
            \draw[dotted] (1,0)--(1,3);
            \draw[dotted] (0,18/5)--(6/5,18/5);
            \draw[dotted] (6/5,0)--(6/5,18/5);
            \draw[dotted] (18/7,0)--(18/7,18/7);
            \draw[dotted] (0,18/7)--(18/7,18/7);
        \end{tikzpicture}
\end{center}

For points $X_j \in \mathbb{R}^2$, $j=1,2,3,4$, we write
\begin{align*}
[X_1,X_2] & = \{ (1-t)X_1 + tX_2 : t \in [0,1]\} \\
[X_1,X_2) & = \{ (1-t)X_1 + tX_2 : t \in [0,1)\} \\
(X_1,X_2) & = \{ (1-t)X_1 + tX_2 : t \in (0,1)\}
\end{align*}
for line segments connecting $X_1$ and $X_2$, including or excluding $X_1$ and $X_2$ as appropriate. We write
$
X_1X_2X_3
$
for the convex hull of $X_1, X_2, X_3$, and
$
{\rm int}\, X_1X_2X_3
$
for the interior of $X_1X_2X_3$. Similarly,
$
X_1X_2X_3X_4
$
denotes the convex hull of $X_1, X_2, X_3,X_4$, and
$
{\rm int}\, X_1X_2X_3X_4
$
denotes the interior of $X_1X_2X_3X_4$. In particular,
\[
{\rm int}\, OAB = \bigg\{ \bigg(\frac{1}{p},\frac{1}{q}\bigg) \in (0,1)^2 : \text{$\frac{1}{q} < \frac{N}{(N-1)p}$ and $\frac{2}{q} + \frac{N}{p} < N$}\bigg\}
\]
and
\[
{\rm int}\, OCDA = \bigg\{ \bigg(\frac{1}{p},\frac{1}{q}\bigg) \in (0,1)^2 : \text{$\frac{1}{q} >\frac{N}{(N-1)p}$ and $\frac{2}{q} + \frac{N}{p} < N$}\bigg\}.
\]


Throughout the paper, we need the exponent $\alpha^*(p,q)$ determined by
\[
\frac{N}{\alpha^*(p,q)}=\frac{1}{q}+\frac{N}{p}.
\]
Note that if $\frac{2}{q} + \frac{N}{p} = N$ (corresponding to the case $s=0$) then $\alpha^*(p,q) = \frac{2p}{p+1}$, which is the exponent in Theorem \ref{M-S}. Also, note that if $(\frac{1}{p},\frac{1}{q})$ belongs to the line segment $[O,A]$, then $\alpha^*(p,q) = q$, and if $(\frac{1}{p},\frac{1}{q})$ belongs to the line segment $[O,B]$, then $\alpha^*(p,q) = p$. In addition, if if $(\frac{1}{p},\frac{1}{q})$ belongs to the region int $OAF$, then $\alpha^*(p,q)<q<p$.
\end{notation}
\subsection{Main results}First, we will prove the following frequency localized estimates of the form \eqref{abstract} for the Dunkl-Schr\"odinger propagator $e^{it\Delta_\kappa}$. In the following statement, $P$ is the operator defined by $\mathcal{F}_\kappa (Pf)(y)=\psi(|y|)\mathcal{F}_\kappa f(y)$, where $\mathcal{F}_\kappa$ is the Dunkl transform introduced in Subsection \ref{Dunkl transform} and $\psi$ is any smooth function on $\mathbb{R}$ supported in $[\frac{1}{2},2]$.
\begin{theorem} \label{frequency-localized} Suppose $N>2$.\\

$(1)$ If $(\frac{1}{p},\frac{1}{q})$ belongs to $ OAB\setminus A$  and $\alpha=\alpha^*(p,q)$, then
\begin{equation}\label{OAB}
			\bigg\|\sum_{j=1}^\infty \lambda_j|e^{it\Delta_\kappa}Pf_j|^2\bigg\|_{L^q(\mathbb{R}, L_\kappa^p(\mathbb{R}^n))}\lesssim_\psi \|\{\lambda_j\}^\infty_{j=1}\|_{\ell^{\alpha}} 
\end{equation}	
  holds for all families of orthonormal functions $\{f_j\}_{j=1}^\infty$ in $L_\kappa^2(\mathbb{R}^n)$ and all sequences $\{\lambda_j\}^\infty_{j=1}$ in $\ell^{\alpha}$. \\
  
$(2)$ If $(\frac{1}{p},\frac{1}{q})$ belongs to int $OCDA$ and $\alpha=q$, then the results \eqref{OAB} hold true.
\end{theorem}

For $1\leq N\leq 2$, the corresponding frequency-localized estimates are provided in the following remarks.
\begin{itemize}
    \item \label{N=2} Let $N=2$.  Then the points $C$ and $D$ coincide. If $(\frac{1}{p},\frac{1}{q})$ belongs to $ OAB\setminus A$  and $\alpha=\alpha^*(p,q)$, the results \eqref{OAB} hold true. Moreover,  if $(\frac{1}{p},\frac{1}{q})$ belongs to int $OAC$ with $\alpha<q$, the results \eqref{OAB} still hold.
    \item \label{1<N<2} For $1<N<2$, if $(\frac{1}{p},\frac{1}{q})$ belongs to $ OAB\setminus A$  and $\alpha=\alpha^*(p,q)$, then the results in equation \eqref{OAB} remain valid.
    \item \label{N=1} When  $N=1$, then the points $A$, $E$ and $G$ coincide. If  $(\frac{1}{p},\frac{1}{q})$ belongs to $ OAB$ but not on the segmen $[A,O),$  and $\alpha=\alpha^*(p,q)$, then the statement in \eqref{OAB} holds true.
\end{itemize}

It is worth noting that along the segment 
$[B, A)$,  the estimates provided in Theorem \ref{frequency-localized} remain valid based on the results of Frank–Sabin in \cite{FS}, even though the operator 
$P$ does not necessarily preserve orthonormality. However, along the segment 
$[C, D]$,  the estimates hold trivially with $\alpha=1$.  Therefore, to establish \eqref{OAB}, it is enough to verify it along the critical segment 
 $[O, A).$ The estimates in Theorem \ref{frequency-localized} along the critical segment $[O, A)$ are subtle, and we establish them using bilinear real interpolation, relying on the endpoint result provided in \cite{KT}.

While the frequency-localized estimates in Theorem \ref{frequency-localized} do not readily extend to general data, as they do in the single-function case, the following strong-type global frequency estimates can be obtained using the restricted weak-type global estimates  (see Proposition \ref{upgrading} in Subsection \ref{L-P section}) and through a series of successive real and complex interpolation arguments.
\begin{theorem}\label{globally}
$(1)$ Let $N\geq 1$. If $(\frac{1}{p},\frac{1}{q})$ belongs to int $OAB$, $\alpha=\alpha^*(p,q)$ and $2s=N-\left(\frac{2}{q}+\frac{N}{p}\right)$, then
\begin{equation*}
			\bigg\|\sum_{j=1}^\infty \lambda_j|e^{it\Delta_\kappa}f_j|^2\bigg\|_{L^q(\mathbb{R}, L_\kappa^p(\mathbb{R}^n))}\lesssim \|\{\lambda_j\}^\infty_{j=1}\|_{\ell^{\alpha}} 
\end{equation*}	
  holds for all families of orthonormal functions $\{f_j\}_{j=1}^\infty$ in $\dot{H}_\kappa^{s}(\mathbb{R}^n)$ and all sequences $\{\lambda_j\}^\infty_{j=1}$ in $\ell^{\alpha}$. 
  
$(2)$ Let $N\geq 2$. If $(\frac{1}{p},\frac{1}{q})$ belongs to int $OCDA$, $2s=N-\left(\frac{2}{q}+\frac{N}{p}\right)$ and $\alpha<q$, then
\begin{equation*}
			\bigg\|\sum_{j=1}^\infty \lambda_j|e^{it\Delta_\kappa}f_j|^2\bigg\|_{L^q(\mathbb{R}, L_\kappa^p(\mathbb{R}^n))}\lesssim_\phi \|\{\lambda_j\}^\infty_{j=1}\|_{\ell^{\alpha}} 
\end{equation*}	
  holds for all families of orthonormal functions $\{f_j\}_{j=1}^\infty$ in $\dot{H}_\kappa^{s}(\mathbb{R}^n)$ and all sequences $\{\lambda_j\}^\infty_{j=1}$ in $\ell^{\alpha}$.  
\end{theorem}
\begin{remark} The key novelty of this paper can be seen as follows:
	\begin{itemize}
		\item 
		The Hardy-Littlewood-Sobolev inequality plays a crucial role in the proof of Theorem  \ref{Strichartz-single}. To sharpen the inequalities in Theorem \ref{Strichartz-single} for functions in Lebesgue spaces, we have to employ its refined form for functions in Lorentz spaces.
		\item 	
 In the proof of  the frequency localized estimates stated in 
 Theorem \ref{frequency-localized}, for a function $\Psi$ defined as  $\mathcal{F}_\kappa \Psi(\xi)=\psi^2(|\xi|)$, one of the crucial step is to observe that applying the Dunkl translation to  $e^{i(t-s)\Delta_\kappa} \Psi(-y)$ yields an estimate of the form
		\begin{equation*}
			\left|\tau_xe^{i(t-s)\Delta_\kappa} \Psi(-y)\right|\leq \frac{C_{\psi,M}}{\left(1+\big||x|-|y|\big|\right)^M}.
		\end{equation*}
  In contrast to the classical Laplacian setting, where the Young's convolution inequality can be applied, to deal with the integral 
		\begin{align*}	\int_{\mathbb{R}^{n}}\int_{\mathbb{R}^{n}}\frac{|W_1(x,t)|^2|W_2(y,s)|^2 }{\left(1+\big||x|-|y|\big|\right)^{M}}h_\kappa(x)h_\kappa(y)dxdy, 
		\end{align*}
		for  $W_1, W_2\in  L^{4}(\mathbb{R}, L_\kappa^{r}(\mathbb{R}^n))$, 
		we require a refined inequality analogous to Young's convolution inequality, which can be expressed as:
		Let $1\leq p,q,r\leq \infty$ satisfy $1+\frac{1}{r}=\frac{1}{p}+\frac{1}{q}$. Suppose $f=\tilde{f}(|\cdot|)$ is a radial function in $L_\kappa^p(\mathbb{R}^n)$ and $g\in L_\kappa^q(\mathbb{R}^n)$, then the function
		\begin{equation*}
			G(x)=\int_{\mathbb{R}^n}\tilde{f}\left(\big||x|-|y|\big|\right)g(y)h_\kappa(y) dy,
		\end{equation*}
		is a radial function in $L_\kappa^r(\mathbb{R}^n)$ satisfying
		\begin{equation*}
			\|G\|_{L_\kappa^r(\mathbb{R}^n)}\leq \|f\|_{L_\kappa^p(\mathbb{R}^n)} \|g\|_{L_\kappa^q(\mathbb{R}^n)}.  
		\end{equation*}

		\item 		Since the Dunkl Laplacian $\Delta_\kappa$ coincides with the classical Euclidean Laplacian $\Delta$ when $\kappa = 0$, our results extend the established orthonormal Strichartz inequalities for the Schrödinger propagator $e^{it\Delta}$ with initial data in Sobolev spaces, as proven by Bez-Hong-Lee-Nakamura-Sawano  in \cite{Bez-Hong-Lee-Nakamura-Sawano}.
     
	\end{itemize}
	
	\end{remark}

Apart from the introduction, the paper is organized as follows: 
\begin{itemize}
	\item In Section \ref{sec2}, we recall the harmonic analysis associated with the Dunkl operators.  In particular, we also recall some important properties of the Dunkl-Schr\"{o}dinger group $e^{it\Delta_\kappa}$, Lorentz spaces, some duality principles, and Littlewood-Paley projections associated with Dunkl operators. 
	\item  In Section \ref{sec3}, we point out some technical lemmas that are extremely important for the proof of our main results of this paper.
	\item   In Section \ref{sec4}, we investigate some refinements of known orthonormal Strichartz estimates but in Lorentz spaces. 
	\item    In Section \ref{sec5}, we prove frequency localized estimates for the Dunkl-Schr\"odinger propagator $e^{it\Delta_\kappa}$.
	\item   Section \ref{sec6} is devoted to proving the strong-type frequency global estimates by using the restricted weak-type global estimates  and a
	succession of  real and complex interpolation arguments.
\end{itemize}  
\noindent \textbf{Notations:}
\begin{itemize}
\item  $A\lesssim B$ means that $A\leqslant CB$, and $A\sim B$ stands for $C_1B\leq A\leq C_2B$, where $C$, $C_1$, $C_2$ denote positive universal constants.
\item $\Delta_\kappa$ denotes Dunkl-Laplacian on $\mathbb{R}^n$.
\item $ P_k$ denotes the Littlewood-Paley operator related to $\Delta_\kappa$.
\item $J_\nu$ denotes the Bessel
function of order $\nu.$ 
\item    $H_\kappa^{p,s}$ denotes the homogeneous Dunkl-Sobolev space.
\end{itemize}

\section{Preliminaries}\label{sec2}
In this section, we recall some basic definitions and important properties of the Dunkl operators and briefly overview the related harmonic analysis to make the paper self-contained. A complete account of harmonic analysis related to Dunkl operators can be found in \cite{ros,ben,Ratna3, Mejjaoli2009,MS,Jiu-Li}.  
	\subsection{Dunkl operators}\label{dunklop} The basic ingredient in the theory of Dunkl operators are root systems and finite reflection groups. We start this section with the definition of a root system. 
 
 Let $\langle\cdot, \cdot\rangle$  denote the standard Euclidean scalar product in $\mathbb{R}^{n}$. For $x \in \mathbb{R}^{n}$, we denote $|x|$ as $|x|=\langle x, x\rangle^{1 / 2}$.
	For $\alpha \in \mathbb{R}^{n} \backslash\{0\},$ we denote $r_{\alpha}$ as the reflection with respect to the hyperplane $\langle \alpha\rangle^{\perp}$, orthogonal to $\alpha$ and is defined by
	$$
	r_{\alpha}(x):=x-2 \frac{\langle \alpha, x\rangle}{|\alpha|^{2}} \alpha, \quad x \in \mathbb{R}^n.
	$$
	A finite set $\mathcal{R}$ in $\mathbb{R}^{n} \backslash\{0\}$ is  said to be a root system if the following holds:
	\begin{enumerate}
		\item $ r_{\alpha}(\mathcal{R})=\mathcal{R}$ for all $\alpha \in \mathcal{R}$,
		\item $\mathcal{R} \cap \mathbb{R} \alpha=\{\pm \alpha\}$ for all $\alpha \in \mathcal{R}$.
	\end{enumerate}
	For a given root system $\mathcal{R}$, the subgroup $G \subset O(n, \mathbb{R})$ generated by the reflections $\left\{r_{\alpha} \mid \alpha \in \mathcal{R}\right\}$ is called the finite Coxeter group associated with $\mathcal{R}$.  The dimension of $span \mathcal{R}$ is called the rank of $\mathcal{R}$.   For a detailed study on the theory of finite reflection groups, we refer to \cite{hum}. Let $\mathcal{R}^+:=\{\alpha\in\mathcal{R}:\langle\alpha,\beta\rangle>0\}$ for some $\beta\in\mathbb{R}^n\backslash\bigcup_{\alpha\in\mathcal{R}}\langle \alpha\rangle^{\perp}$,  be a fixed positive root system.

	Some typical examples of such a system are the Weyl groups such as the symmetric group $S_{n}$ for the type $A_{n-1}$ root system and the hyperoctahedral group for the type $B_{n}$ root system. In addition, $H_{3}, H_{4}$ (icosahedral groups), and $I_{2}(n)$ (symmetry group of the regular $n$-gon) are also Coxeter groups. 

	A multiplicity function for $G$ is a function $\kappa: \mathcal{R} \rightarrow \mathbb{C}$ which is constant on $G$-orbits. Setting $\kappa_{\alpha}:=\kappa(\alpha)$ for $\alpha \in \mathcal{R},$  from the definition of $G$-invariant,  we have $\kappa_{g \alpha}=\kappa_{\alpha}$ for all $g \in G$.  We say $\kappa$ is non-negative if $\kappa_{\alpha} \geq 0$ for all $\alpha \in \mathcal{R}$. The $\mathbb{C}$-vector space of non-negative multiplicity functions on $\mathcal{R}$ is denoted by $\mathcal{K}^{+}$. Let us denote  $\gamma$ as $\gamma=\gamma(\kappa):=\sum\limits_{\alpha\in\mathcal{R}^+}\kappa(\alpha)$. For convenience, we denote $N=2\gamma+n$ throughout the paper. 
 
 For $\xi \in \mathbb{C}^{n}$ and $\kappa \in \mathcal{K}^{+},$ Dunkl  in 1989 introduced a family of first-order differential-difference operators $T_{\xi}:= T_{\xi}(\kappa)$  by
	\begin{align}\label{dunkl}
		T_{\xi}(\kappa) f(x):=\partial_{\xi} f(x)+\sum_{\alpha \in \mathcal{R}^{+}} \kappa_{\alpha }\langle \alpha, \xi\rangle \frac{f(x)-f\left(r_{\alpha} x\right)}{\langle \alpha, x\rangle}, \quad f \in C^{1}\left(\mathbb{R}^{n}\right),
	\end{align}
	where $\partial_{\xi}$ denotes the directional derivative corresponding to $\xi$.  The operator $T_\xi$   defined in (\ref{dunkl}), formally known as Dunkl operator and is one of the most important developments in the theory of special functions associated with root systems \cite{dun}. They commute pairwise and are skew-symmetric with respect to the $G$-invariant measure $h_\kappa(x)dx$, where the weight function $$h_\kappa(x):=\prod\limits_{\alpha\in \mathcal{R}^{+}}|\langle \alpha, x\rangle|^{2 \kappa_{\alpha}}$$   is of 
 homogeneous degree $2\gamma$. Thanks to the $G$-invariance of the multiplicity function, the definition of the Dunkl operator is independent of the choice of the positive subsystem $\mathcal{R}^{+}$.  In \cite{dun1991},	it is shown that for any $\kappa\in\mathcal{K}^+$, there is a unique linear isomorphism $V_\kappa$ (Dunkl's intertwining operator) on the space $\mathcal{P}(\mathbb{R}^n)$ of polynomials on $\mathbb{R}^n$ such that
	\begin{enumerate}
		\item $V_\kappa\left(\mathcal{P}_m(\mathbb{R}^n)\right)=\mathcal{P}_m(\mathbb{R}^n)$ for all $m\in\mathbb{N}$,
		\item $V_\kappa|_{\mathcal{P}_0(\mathbb{R}^n)}=id$,
		\item $T_\xi(\kappa)V_\kappa=V_\kappa\partial_\xi$,
	\end{enumerate}
	where  $\mathcal{P}_m(\mathbb{R}^n)$ denotes the space of homogeneous polynomials of degree $m$. 
	
	For any finite reflection group $G$ and   any $k\in\mathcal{K}^+$, R\"{o}sler in \cite{ros} proved that there exists a unique positive Radon probability measure $\rho_x^\kappa$ on $\mathbb{R}^n$ such that
	\begin{equation}\label{239}	V_\kappa f(x)=\int_{\mathbb{R}^n}f(\xi)d\rho_x^\kappa(\xi).
	\end{equation}
	The measure $\rho_x^k$ depends on $x\in\mathbb{R}^n$ and its support is contained in the ball $B\left(\|x\|\right):=\{\xi\in\mathbb{R}^n: \|\xi\|\leq\|x\|\}$. In view of the Laplace type representation \eqref{239}, Dunkl's intertwining operator $V_\kappa$ can be extended to a larger class of function spaces.
	
	Let $\{\xi_1, \xi_2, \cdots, \xi_n\}$ be an orthonormal basis of $(\mathbb{R}^n , \langle \cdot, \cdot\rangle )$. Then  the Dunkl Laplacian operator $\Delta_\kappa $ is defined as $$\Delta_\kappa=\sum_{j=1}^nT^2_{\xi_j}(\kappa).$$
	The definition of  $\Delta_\kappa$ is independent of the choice of the orthonormal basis of $\mathbb{R}^n.$ In fact,   one can see that the operator $\Delta_\kappa $ also can be expressed as $$\Delta_\kappa f(x)=\Delta f(x)+\sum_{\alpha \in \mathcal{R}^+}\kappa_\alpha \left\{\frac{2 \langle \nabla f(x),  \alpha \rangle}{\langle \alpha, x\rangle } -|\alpha|^2 \frac{f(x)-f(r_\alpha x)}{\langle \alpha, x\rangle^2 } \right\}, \quad f\in C^1(\mathbb{R}^n),$$
	where $\nabla$ and $\Delta$ are    the usual gradiant  and   Laplacian operator on $\mathbb{R}^n$, respectively.  Observe that, for $\kappa \equiv 0$, the Dunkl Laplacian operator $\Delta_\kappa$ reduces to the classical Euclidean Laplacian $\Delta$. 

For every $y\in\mathbb{R}^n$, the system 
\begin{equation*}
	\begin{cases}
	T_\xi u(x,y)=\langle y,\xi\rangle\, u(x,y), \quad \xi\in\mathbb{R}^n,\\
	u(0,y)=1,
	\end{cases}
	\end{equation*}
admits a unique analytic solution on $\mathbb{R}^n$, which we denote as  $K(x,y)$ and is known as the Dunkl kernel in general. The kernel has a unique holomorphic extension to $\mathbb{C}^n\times \mathbb{C}^n$ and is given by
\begin{equation*}
K(x,y)=V_\kappa(e^{\langle \cdot,y\rangle})(x)=\int_{\mathbb{R}^n}e^{\langle \xi,y\rangle}d\rho_x^\kappa(\xi),\quad \forall x,y\in\mathbb{R}^n.
	\end{equation*}
When $\kappa \equiv 0$, the Dunkl kernel $K(x,y)$  reduces to the exponential $e^{\langle x,y\rangle}$. For $x,y\in\mathbb{R}^n$ and  $z,w\in\mathbb{C}^n$, the Dunkl kernel satisfies the following properties:
\begin{enumerate}[(i)]
    \item $K(z,w)=K(w,z)$ and $K(\lambda z,w)=K(z,\lambda w),\quad \lambda\in\mathbb{C}$;
    \item  for any $\nu\in\mathbb{N}^n$, we have
$$|D_z^\nu K(x,z)|\leq |x|^{|\nu|}e^{|Re(z)||x|},$$
where $D_z^\nu=\frac{\partial^{|\nu|}}{\partial {z_1}^{\nu_1}\cdots {z_n}^{\nu_n}}$ and $|\nu|=\nu_1+\cdots+\nu_n$. In particular, for any $x,y\in\mathbb{R}^n$, 
$$|K(x,iy)|\leq 1;$$
\item $K(ix,y)=\overline{K(-ix,y)}$ and $K(gx,gy)=K(x,y),\quad g\in G$.
\end{enumerate}

\subsection{Dunkl transform} \label{Dunkl transform}
For $1\leq p<\infty$, let $L_\kappa^p\left(\mathbb{R}^{n}\right)$ be the space of $L^p$-functions on $\mathbb{R}^{n}$ with respect to the weight $h_\kappa(x)$ with the $L_\kappa^p\left(\mathbb{R}^{n}\right)$-norm
\begin{equation*}
\|f\|_{L_\kappa^p\left(\mathbb{R}^{n}\right)}=\left(\int_{\mathbb{R}^{n}}|f(x)|^ph_\kappa(x)dx\right)^\frac{1}{p},
\end{equation*}
and $\|f\|_{L_\kappa^\infty\left(\mathbb{R}^{n}\right)}=ess \sup\limits_{x\in\mathbb{R}^n}|f(x)|$.
 The Dunkl transform is defined on $L_\kappa^1\left(\mathbb{R}^{n}\right)$ by
 \begin{equation}\label{Dunkltransform}
 \mathcal{F}_\kappa f(y)=\frac{1}{c_\kappa}\int_{\mathbb{R}^{n}} f(x) K(x,-iy)h_\kappa(x)dx,
 \end{equation}
 where $c_\kappa$ is the Mehta-type constant defined by
 \begin{equation*}
 c_\kappa=\int_{\mathbb{R}^{n}} e^{-\frac{|x|^2}{2}}h_\kappa(x)dx.
 \end{equation*}
When $\kappa \equiv 0$, the Dunkl transform coincides with the classical Fourier transform. The Dunkl transform plays the same role as the Fourier transform in classical Fourier analysis and enjoys properties similar to those of the classical Fourier transform. 
Here we list some basic properties of the Dunkl transform (see \cite{Dunklll, Jeu, Dai-Ye}).
\begin{enumerate}[(i)]
    \item  For all $f\in L_\kappa^1\left(\mathbb{R}^{n}\right)$, we have
\begin{equation*}
\|\mathcal{F}_\kappa f\|_{L_\kappa^\infty\left(\mathbb{R}^{n}\right)}\leq \frac{1}{c_\kappa} \|f\|_{L_\kappa^1\left(\mathbb{R}^{n}\right)}.
\end{equation*}
\item For any function $f$ in the Schwartz space $\mathcal{S}\left(\mathbb{R}^{n}\right)$, we have
\begin{equation*}
\mathcal{F}_\kappa(T_\xi f)(y)=i\langle \xi, y\rangle \mathcal{F}_\kappa f(y),\quad y,\xi\in\mathbb{R}^{n}.
\end{equation*}
In particular, it follows that 
\begin{equation*}
\mathcal{F}_\kappa(\Delta_\kappa f)(y)=-|y|^2 \mathcal{F}_\kappa f(y),\quad y\in\mathbb{R}^{n}.
\end{equation*}
\item  For all $f\in L_\kappa^1\left(\mathbb{R}^{n}\right)$, we have
\begin{equation*}
\mathcal{F}_\kappa (f(\cdot/\lambda))(y)=\lambda^N\mathcal{F}_\kappa f(\lambda y),\quad y\in\mathbb{R}^{n},
\end{equation*}for all $\lambda>0$.
\item The Dunkl transform $\mathcal{F}_\kappa$ is a homeomorphism of the Schwartz space $\mathcal{S}\left(\mathbb{R}^{n}\right)$ and its inverse is given by $\mathcal{F}_\kappa^{-1}g(x)=\mathcal{F}_\kappa g(-x),$ for all $ g\in \mathcal{S}\left(\mathbb{R}^{n}\right)$. In addition, for all $f\in \mathcal{S}\left(\mathbb{R}^{n}\right)$, it satisfies
\begin{equation*}
 \int_{\mathbb{R}^{n}} |f(x)|^2 h_\kappa(x)dx=\int_{\mathbb{R}^{n}} |\mathcal{F}_\kappa f(y)|^2 h_\kappa(y)dy.
 \end{equation*}
In particular, the Dunkl transform extends to an isometric isomorphism on $L_\kappa^2\left(\mathbb{R}^{n}\right)$. The Dunkl transform can be extended to the space of tempered distributions $\mathcal{S}'\left(\mathbb{R}^{n}\right)$ and is also a homeomorphism of $\mathcal{S}'\left(\mathbb{R}^{n}\right)$.
\item For every $f\in L_\kappa^1\left(\mathbb{R}^{n}\right)$ such that $\mathcal{F}_\kappa f\in L_\kappa^1\left(\mathbb{R}^{n}\right)$, we have the inversion formula
\begin{equation*}
f(x)=\frac{1}{c_\kappa}\int_{\mathbb{R}^{n}} \mathcal{F}_\kappa f(y) K(ix,y)h_\kappa(y)dy,\quad  a.e. \;x\in\mathbb{R}^{n}.
\end{equation*}
\item  If $f$ is a radial function in $ L_\kappa^1\left(\mathbb{R}^{n}\right)$ such that $f(x)=\tilde{f}(|x|)$, then
\begin{equation}\label{radial-Dunkl}
\mathcal{F}_\kappa f(y)=\frac{1}{\Gamma(N/2)}\int_0^\infty \tilde{f}(r)\frac{J_\frac{N-2}{2}(r|y|)}{\left(r|y|\right)^\frac{N-2}{2}}r^{N-1} dr,
\end{equation}
where $J_\nu$ denotes the Bessel function of order $\nu>-\frac{1}{2}$.

\item For $1\leq p\leq \infty$ and $s\in\mathbb{R}$, the homogeneous Dunkl-Sobolev space $\dot{H}_\kappa^{p,s}\left(\mathbb{R}^{n}\right)$ is the set of tempered distributions $u \in \mathcal{S}'\left(\mathbb{R}^{n}\right)$ such that $\mathcal{F}^{-1}_\kappa\left(|\cdot|^s \mathcal{F}_\kappa f\right)\in L_\kappa^p\left(\mathbb{R}^{n}\right)$ and is equiped with the norm 
\begin{equation*}
\|u\|_{\dot{H}_\kappa^{p,s}\left(\mathbb{R}^{n}\right)}=\left\|\mathcal{F}^{-1}_\kappa\left(|\cdot|^s \mathcal{F}_\kappa f\right)\right\|_{L_\kappa^p\left(\mathbb{R}^{n}\right)}.
\end{equation*}
In particular, when $p=2$, we denote $\dot{H}_\kappa^{p,s}\left(\mathbb{R}^{n}\right)$ by $\dot{H}_\kappa^{s}\left(\mathbb{R}^{n}\right)$ in short and $\dot{H}_\kappa^{0}\left(\mathbb{R}^{n}\right)=L_\kappa^2\left(\mathbb{R}^{n}\right)$.\\
 For any $1\leq p_1\leq p_2\leq\infty$ and $s_1\leq s_2 \in\mathbb{R}$ such that $s_1-\frac{N}{p_1}=s_2-\frac{N}{p_2}$, we also have the  continuous inclusion:  
\begin{equation}\label{Sobolev-Embedding}
\dot{H}_\kappa^{p_1,s_1} (\mathbb{R}^n)\subseteq \dot{H}_\kappa^{p_2,s_2} (\mathbb{R}^n).
\end{equation}

\end{enumerate}

\subsection{Dunkl convolution operator} For  given $x\in\mathbb{R}^n$, the Dunkl translation operator $f\mapsto \tau_x f$ is defined on $\mathcal{S}\left(\mathbb{R}^{n}\right)$ by
\begin{equation*}
\mathcal{F}_\kappa(\tau_x f)(y)=K(x,iy)\mathcal{F}_\kappa f(y),\quad y\in\mathbb{R}^n.
\end{equation*}
At the moment, the explicit formula of the Dunkl translation operator is known only in two cases. One is when $f(x)=\tilde{f}(|x|)$ is a continuous radial function in $L_\kappa^2(\mathbb{R}^n)$, the Dunkl translation operator is represented by (see \cite{ros, Dai-Wang})
\begin{equation} \label{translator}
\tau_xf(y)=V_\kappa[\tilde{f}\left(\sqrt{|y|^2+|x|^2-2\langle y,\cdot\rangle}\right)](x)=\int_{\mathbb{R}^n}\tilde{f}\left(\sqrt{|y|^2+|x|^2-2\langle y,\xi\rangle}\right)d\rho^\kappa_x(\xi).
\end{equation}
The Dunkl translation operator can also be extended to the space of tempered distributions $\mathcal{S}'\left(\mathbb{R}^{n}\right)$. It turns out to be rather difficult to extend important results in classical Fourier analysis to the setting of the Dunkl transform in general. One of the difficulties comes from the fact that the Dunkl translation operator is not positive in
general. In fact, even the $L_\kappa^p(\mathbb{R}^n)$-boundedness of $\tau_x$ is not established in general. On the other hand, when $f$ is a radial function in $L_\kappa^p(\mathbb{R}^n)$, $1\leq p\leq \infty$, the following holds (see \cite{GIT2019, Thangavelu-Xu2005})
\begin{equation*}
\|\tau_xf\|_{L_\kappa^p(\mathbb{R}^n)}\leq \|f\|_{L_\kappa^p(\mathbb{R}^n)}.
\end{equation*}
Using the Dunkl translation operator, we define the Dunkl convolution  of functions $f,g\in \mathcal{S}\left(\mathbb{R}^{n}\right)$ by
\begin{equation*}
f*_\kappa g(x)=\int_{\mathbb{R}^n}\tau_xf(-y)g(y)h_\kappa(y)dy, \quad x\in\mathbb{R}^n.
\end{equation*}
The Dunkl convolution satisfies the following properties (see \cite{Thangavelu-Xu2005}).
\begin{enumerate}
    \item  $\mathcal{F}_\kappa(f*_\kappa g)=\mathcal{F}_\kappa (f)\mathcal{F}_\kappa (g)$, $\mathcal{F}^{-1}_\kappa(f*_\kappa g)=\mathcal{F}^{-1}_\kappa (f)\mathcal{F}^{-1}_\kappa (g)$, and $f*_\kappa g=g*_\kappa f$.
\item Young's inequality: Let $1\leq p,q,r\leq\infty$ be such that $1+\frac{1}{r}=\frac{1}{p}+\frac{1}{q}$. If $f\in L^p_\kappa(\mathbb{R}^n)$ and $g$ is a radial function of $L_\kappa^q(\mathbb{R}^n)$, then $f*_\kappa g\in L_\kappa^r(\mathbb{R}^n)$ and we have
\begin{equation}\label{Young}
\|f*_\kappa g\|_{L_\kappa^r(\mathbb{R}^n)}\leq \|f\|_{L_\kappa^p(\mathbb{R}^n)} \|g\|_{L_\kappa^q(\mathbb{R}^n)}.
\end{equation}
\end{enumerate}

\subsection{Dunkl-Schr\"{o}dinger semigroup}
The Dunkl-Schr\"{o}dinger group $e^{it\Delta_\kappa}$ is unitary on $L^2_\kappa(\mathbb{R}^n)$ define by
\begin{equation*}
e^{it\Delta_\kappa}f(x)=\mathcal{F}^{-1}_\kappa\left(e^{-it|\cdot|^2} \mathcal{F}_\kappa f\right)(x)=\mathcal{F}^{-1}_\kappa\left(e^{-it|\cdot|^2}\right)*_\kappa f(x)=K_{it}*_\kappa f(x),
\end{equation*}
where the Dunkl-Schr\"{o}dinger kernel
\begin{equation}\label{Schrodinger}
K_{it}(x)=\mathcal{F}^{-1}_\kappa\left(e^{-it|\cdot|^2}\right)(x)=\frac{1}{c_\kappa |t|^\frac{N}{2}}e^{-i\frac{N\pi}{4} sgn t}e^{i\frac{|x|^2}{4t}}.
\end{equation}
Noting that the Dunkl-Schr\"{o}dinger kernel $K_{it}$ is radial, from \eqref{translator}, we have 
\begin{equation}\label{translator-decay}
\left|\tau_x K_{it}(y)\right|\leq \frac{1}{c_\kappa |t|^\frac{N}{2}},\forall x,y\in\mathbb{R}^n.
\end{equation}
If $f\in L_\kappa^1(\mathbb{R}^n)$, we have the dispersive estimate
\begin{equation*}
\|e^{it\Delta_\kappa}f\|_{L_\kappa^\infty(\mathbb{R}^n)}\leq \frac{1}{c_\kappa |t|^\frac{N}{2}}\|f\|_{L_\kappa^1(\mathbb{R}^n)}.
\end{equation*}
As the Dunkl-Schr\"{o}dinger operator $e^{it\Delta_\kappa}$ is unitary on $L^2(\mathbb{R}^n)$, combined with the above dispersive estimate,  for any $2\leq p\leq \infty$, interpolation gives
\begin{equation}\label{dispersive}
    \|e^{it\Delta_\kappa}g\|_{L_\kappa^{p}(\mathbb{R}^n)}\lesssim \frac{1}{|t|^\frac{N}{2p^\prime}} \|g\|_{L_\kappa^{p^\prime}(\mathbb{R}^n)}.
\end{equation}
By Theorem \ref{Strichartz-single}, for $p,q\geq1$, $(q,p,N)\neq (1,\infty,2)$ and $\frac{2}{q}+\frac{N}{p}=N$, the Dunkl-Schr\"{o}dinger operator $e^{it\Delta_\kappa}$ is bounded from $L_\kappa^2(\mathbb{R}^n)$ to $L^{2q}(\mathbb{R}, L_\kappa^{2p}(\mathbb{R}^n))$. By duality, the adjoint operator $\left(e^{it\Delta_k}\right)^*$ of the Dunkl-Schr\"{o}dinger operator $e^{it\Delta_\kappa}$ defined by
\begin{equation*}
\left(e^{it\Delta_k}\right)^*g(x)=\int_{\mathbb{R}}e^{-it\Delta_\kappa} g(x,t)dt,
\end{equation*}
is bounded from $L^{(2q)^\prime}(\mathbb{R}, L_\kappa^{(2p)^\prime}(\mathbb{R}^n))$ to $L_\kappa^2(\mathbb{R}^n)$ and the composition of the maps $e^{it\Delta_\kappa}$ and $\left(e^{it\Delta_k}\right)^*$ defined by
\begin{equation*}
e^{it\Delta_k}\left(e^{it\Delta_k}\right)^*g(x)=\int_{\mathbb{R}}e^{i(t-s)\Delta_\kappa} g(x,s)ds=\int_\mathbb{R} \int_{\mathbb{R}^n}\tau_xK_{i(t-s)}(-y)g(y,s)h_\kappa(y)dyds
\end{equation*}
is bounded from $L^{(2q)^\prime}(\mathbb{R}, L_\kappa^{(2p)^\prime}(\mathbb{R}^n))$ to $L^{2q}(\mathbb{R}, L_\kappa^{2p}(\mathbb{R}^n))$.

\begin{remark}
The Dunkl kernel of the localised  Dunkl-Schr\"{o}dinger operator $e^{it\Delta_\kappa}P$ is also radial and given by $\mathcal{F}^{-1}_\kappa (e^{-it|\cdot|^2} \psi(|\cdot|))$ satisfying the following dispersive estimate, i.e.,
\begin{equation}\label{cut-off decay}
   \left|\tau_x\mathcal{F}^{-1}_\kappa \left(e^{-it|\cdot|^2} \psi(|\cdot|)\right)(y)\right|\lesssim_\psi |t|^{-\frac{N}{2}}, \quad \forall x,y\in\mathbb{R}^n,
\end{equation}
which follows from Young's inequality \eqref{Young}, \eqref{translator} and \eqref{Schrodinger}.
\end{remark}

\subsection{Lorentz spaces}
In this subsection, we introduce the Lorentz space $L_\mu^{p,r}(\mathbb{R}^n)$ and the Lorentz sequence space $\ell^{p,r}$, where $\mu$ is a positive measure on $\mathbb{R}^n$. We refer the reader to \cite{Stein-Weiss} for further details.

First, let us recall some facts about the rearrangement argument. Suppose $f$ is a complex-valued $\mu$-measurable function on $\mathbb{R}^n$ such that $\mu\left(\{x\in\mathbb{R}^n: |f(x)|>t\}\right)=\int_{\{x\in\mathbb{R}^n: |f(x)|>t\}}d\mu(x)<+\infty$ for every $t>0$. The decreasing rearrangement function $f^*$ of $f$ is defined by
\begin{equation*}
f^*(t)=\inf\{s>0: a_f(s)\leq t\},
\end{equation*}
where $a_f$ is the distribution function of $f$ defined by
\begin{equation*}
a_f(t)=\mu\left(\{x\in\mathbb{R}^n: |f(x)|>t\}\right).
\end{equation*}
Since $f^*$ is decreasing, the maximal function $f^{**}$ of $f^*$ defined by
\begin{equation*}
f^{**}(t)=\frac{1}{t}\int_0^t f^*(s)ds, 
\end{equation*}
is also decreasing and $f^*(t)\leq f^{**}(t)$ for every $t\geq 0$.

For $1<p<\infty$ and $1\leq r\leq \infty$, the Lorentz space $L^{p,r}_\mu(\mathbb{R}^n)$ is defined as the set of $\mu$-measurable functions $f$ on $\mathbb{R}^n$ such that $\|f\|_{L^{p,r}_\mu(\mathbb{R}^n)}<\infty$, where
\begin{equation*}
\|f\|_{L^{p,r}_\mu(\mathbb{R}^n)}=
\begin{cases}
\left(\int\limits_0^\infty (t^\frac{1}{p}f^*(t))^r\frac{dt}{t}\right)^\frac{1}{r},&\text{ if } 1\leq r<\infty,\\
\sup\limits_{t>0}t^\frac{1}{p}f^*(t),&\text{ if } r=\infty.
\end{cases}
\end{equation*}
It is well-known that the function $\|\cdot\|_{L^{p,r}_\mu(\mathbb{R}^n)}$ defined above is a norm when $r\leq p$ and a quasi-norm otherwise. In order to obtain a norm in all cases, we define
\begin{equation*}
    \|f\|^*_{L^{p,r}_\mu(\mathbb{R}^n)}=
\begin{cases}
\left(\int\limits_0^\infty (t^\frac{1}{p}f^{**}(t))^r\frac{dt}{t}\right)^\frac{1}{r},&\text{ if } 1\leq r<\infty,\\
\sup\limits_{t>0}t^\frac{1}{p}f^{**}(t),&\text{ if } r=\infty.
\end{cases}
\end{equation*}
One can prove that   $\|\cdot\|^*_{L^{p,r}_\mu(\mathbb{R}^n)}$ is a norm, which is equivalent to $\|\cdot\|_{L^{p,r}_\mu(\mathbb{R}^n)}$ in the sense that
\begin{equation*}
\|f\|_{L^{p,r}_\mu(\mathbb{R}^n)}\leq\|f\|^*_{L^{p,r}_\mu(\mathbb{R}^n)}\leq \frac{p}{p-1}\|f\|_{L^{p,r}_\mu(\mathbb{R}^n)}, \;\forall f\in L^{p,r}_\mu(\mathbb{R}^n).
\end{equation*}
It is clear that $L^{p,p}_\mu(\mathbb{R}^n)=L^p_\mu(\mathbb{R}^n)$ and the Lorentz spaces are monotone with respect to second exponent, namely
\begin{equation}\label{embedding}
    L^{p,r_1}_\mu(\mathbb{R}^n)\subseteq L^{p,r_2}_\mu(\mathbb{R}^n),\;1\leq r_1\leq r_2\leq \infty.
\end{equation}
In this paper, we denote by $\|\cdot\|_{L^{p,r}_\kappa(\mathbb{R}^n)}$ the Lorentz space $\|\cdot\|_{L^{p,r}_{h_\kappa}(\mathbb{R}^n)}$ in short and when $\kappa\equiv 0$, we denote by $\|\cdot\|_{L^{p,r}(\mathbb{R}^n)}$.

Finally, we recall the Lorentz sequence space $\ell^{p,r}$. Let $\{\lambda_j\}^\infty_{j=1}\in c_0$ and $\{\lambda^*_j\}^\infty_{j=1}$ be the sequence permuted in a decreasing order. For $1<p<\infty$ and $1\leq r\leq \infty$, the Lorentz sequence space $\ell^{p,r}$ is defined as the set of all sequences $\{\lambda_j\}^\infty_{j=1}\in c_0$ such that $\|\{\lambda_j\}^\infty_{j=1}\|_{\ell^{p,r}}<\infty$, where
\begin{equation*}
\|\{\lambda_j\}^\infty_{j=1}\|_{\ell^{p,r}}=
\begin{cases}
\left(\sum\limits_{j=1}^\infty (j^\frac{1}{p}\lambda_j^*)^r\frac{1}{j}\right)^\frac{1}{r},&\text{ if } 1\leq r<\infty,\\
\sup\limits_{j\geq1}j^\frac{1}{p}\lambda_j^*,&\text{ if } r=\infty.
\end{cases}
\end{equation*}

\subsection{Schatten class} 		
Let \(\mathcal{H}\) be a complex separable Hilbert space. A linear compact operator \(A : \mathcal{H} \rightarrow \mathcal{H}\) belongs to the $\alpha$-Schatten-von Neumann class \(\mathfrak{S}^\alpha(\mathcal{H})\) if
	$$
	\sum_{j=1}^{\infty}\left(s_j(A)\right)^\alpha<\infty,
	$$where $\{s_j(A)\}_{j=1}^\infty$ denote the singular values of \(A,\), i.e., the eigenvalues of \(|A|=\sqrt{A^{*} A}\)
	with multiplicities counted. For $1 \leq \alpha<\infty,$ the Schatten space $\mathfrak{S}^\alpha(\mathcal{H})$ is defined as the space of all compact operators $A$ on $\mathcal{H}$ such that $\displaystyle \sum_{j=1}^\infty  \left(s_j(A)\right)^\alpha<\infty$, and the class \(\mathfrak{S}^\alpha(\mathcal{H})\) is a Banach space
	endowed with the norm
$$\|A\|_{\mathfrak{S}^\alpha(\mathcal{H})}=\left(\sum_{j=1}^{\infty}\left(s_j(A)\right)^\alpha\right)^{\frac{1}{\alpha}}.
	$$
	For  \(0<\alpha<1\), the $\|\cdot\|_{\mathfrak{S}^\alpha(\mathcal{H})}$ as above only defines a quasi-norm with respect to which
	\(\mathfrak{S}^\alpha(\mathcal{H})\) is complete. For $\alpha=\infty$, \(\mathfrak{S}^\infty(\mathcal{H})\) means the
space of the compact operators rather than the bounded operators. But the norm $\|\cdot\|_{\mathfrak{S}^\infty(\mathcal{H})}$ is the
same as for the bounded operators. An operator belongs to the class \(\mathfrak{S}^{1}(\mathcal{H})\) is known as {\it Trace class} operator. Also, an operator belongs to   \(\mathfrak{S}^{2}(\mathcal{H})\) is known as  {\it Hilbert-Schmidt} operator.

 We shall make use of the following duality principle to derive orthonormal Strichartz inequalities from estimates in terms of Schatten bounds. For the proof, we refer to \cite[Lemma 3]{FS} with appropriate modifications.
 \begin{lemma}[Duality principle]\label{duality-principle} Assume that $A$ is a bounded linear operator from $L_\kappa^2(\mathbb{R}^n)$ to $L^{2q,2r}(\mathbb{R}, L_\kappa^{2p}(\mathbb{R}^n))$ for some $p,q,\alpha\geq 1$. Then 
         \begin{equation*}
			\bigg\|\sum_{j=1}^\infty \lambda_j|A f_j|^2\bigg\|_{L^{q,r}(\mathbb{R}, L_\kappa^{p}(\mathbb{R}^n))}\lesssim \|\{\lambda_j\}^\infty_{j=1}\|_{\ell^{\alpha}} ,
		\end{equation*}	
  holds for all families of orthonormal functions $\{f_j\}_{j=1}^\infty$ in $L_\kappa^2(\mathbb{R}^n)$ and all sequences $\{\lambda_j\}^\infty_{j=1}$ in $\ell^{\alpha}$ if and only if
		\begin{equation*}
	\big\|WAA^*\overline{W}\big\|_{\mathfrak{S}^{\alpha^\prime}(L^2(\mathbb{R}, L_\kappa^2(\mathbb{R}^n)))}\lesssim \big\|W\big\|^2_{L^{2q^\prime,2r^\prime}(\mathbb{R}, L_\kappa^{2p^\prime}(\mathbb{R}^n))},
		\end{equation*}
  for all $W\in L^{2q^\prime,2r^\prime}(\mathbb{R}, L_\kappa^{2p^\prime}(\mathbb{R}^n))$, where the function $W$ is interpreted as an operator which acts by multiplication.
	\end{lemma}
We also need the complex interpolation method in Schatten spaces. For the proof, we refer to \cite[Propostion 1]{FS} with appropriate modifications.
\begin{lemma}\label{interpolation schatten}
		Let $\{T_z\}$ be an analytic family of operators on $\mathbb{R}^d\times \mathbb{R}$ in the sense of Stein defined in the strip $a_0\leq Re (z) \leq a_1$ for some $a_0<a_1$. If there exist $M_0, M_1, b_0, b_1>0$, $1\leq p_0,p_1,q_0,q_1\leq \infty$ and $1\leq \alpha_0,\alpha_1\leq\infty$ such that for all $s\in\mathbb{R}$, one has for all simple functions $W_1,W_2$ on $\mathbb{R}^d\times \mathbb{R}$
		\begin{small}\begin{equation*}		
				\begin{aligned}
					&
 \left\|W_1T_{a_0+is}W_2\right\|_{\mathfrak{S}^{\alpha_0}\left(L^2(\mathbb{R}, L_\kappa^2(\mathbb{R}^d))\right)}\leq M_0 e^{b_0|s|}\left\|W_1\right\|_{L^{q_0}(\mathcal{\mathbb{R}}, L_\kappa^{p_0}(\mathbb{R}^d))}\left\|W_2\right\|_{L^{q_0}(\mathbb{R}, L_\kappa^{p_0}(\mathbb{R}^d))},\\			
					&	\left\|W_1T_{a_1+is}W_2\right\|_{\mathfrak{S}^{\alpha_1}\left(L^2(\mathbb{R}, L_\kappa^2(\mathbb{R}^d))\right)}\leq M_1e^{b_1|s|}\left\|W_1\right\|_{L^{q_1}(\mathbb{R}, L_\kappa^{p_1}(\mathbb{R}^d))}\left\|W_2\right\|_{L^{q_1}(\mathbb{R}, L_\kappa^{p_1}(\mathbb{R}^d))}.
				\end{aligned}
			\end{equation*}
		\end{small}
		Then for all $\theta\in [0,1]$, $a_\theta, \alpha_\theta, p_\theta$ and $q_\theta$ such that
			\begin{equation*}
				a_\theta=(1-\theta)a_0+a_1,\; \frac{1}{\alpha_\theta}=\frac{1-\theta}{\alpha_0}+\frac{\theta}{\alpha_1},\;\frac{1}{p_\theta}=\frac{1-\theta}{p_0}+\frac{\theta}{p_1}\;\text{and}\;\frac{1}{q_\theta}=\frac{1-\theta}{q_0}+\frac{\theta}{q_1},
			\end{equation*}
            we have
			\begin{equation*}
\left\|W_1T_{a_\theta}W_2\right\|_{\mathfrak{S}^{\alpha_\theta}\left(L^2(\mathbb{R}, L_\kappa^2(\mathbb{R}^d))\right)}\leq  M_0^{1-\theta}M_1^\theta\left\|W_1\right\|_{L^{q_\theta}(\mathbb{R}, L_\kappa^{p_\theta}(\mathbb{R}^d))}\left\|W_2\right\|_{L^{q_\theta}(\mathbb{R}, L_\kappa^{p_\theta}(\mathbb{R}^d))}.
			\end{equation*}
	\end{lemma}
 
\subsection{Littlewood-Paley projections and an interpolation method}\label{L-P section}
In this subsection, we recall the Littlewood-Paley decomposition related to the Dunkl Laplacian $\Delta_\kappa$ and introduce an interpolation method which extends frequency localized estimates to general data.

We choose $\phi\in C_c^\infty(\mathbb{R})$ to be supported in $[1/2,2]$ such that 
	\begin{equation*}
\underset{j\in \mathbb{Z}}{\sum}\phi(2^{-j}t)=1, \text{ for all } t>0.
	\end{equation*}
For $f\in \mathcal{S}'\left(\mathbb{R}^{n}\right)$, we define the Littlewood-Paley decomposition related to the Dunkl Laplacian by
\begin{equation*}
    P_kf=\phi(2^{-j}\sqrt{-\Delta_\kappa})f=\mathcal{F}_\kappa^{-1}\left(\phi(2^{-j}|\cdot|)\mathcal{F}_\kappa f\right), \text{ for all }  k\in\mathbb{Z}.
\end{equation*}
 Now proceeding similarly as in \cite{Bez-Hong-Lee-Nakamura-Sawano}, we have the following proposition to upgrade localized estimates to annuli to global estimates in restricted weak-type form.
\begin{proposition}\label{upgrading}
    Let $q_0,q_1>1$, $p,\alpha_0,\alpha_1\geq1$ and $\{g_j\}_{j=1}^\infty$ be a uniformly bounded sequence in $L^{2q_i}(\mathbb{R},L_\kappa^{2p}(\mathbb{R}^n))$ for each $i=0,1$. If for each $i=0,1,$ there exists $\epsilon_i>0$ such that
    \begin{equation*}
        \left\|\sum_{j=1}^\infty \lambda_j|P_kg_j|^2\right\|_{L^{q_i,\infty}(\mathbb{R},L_\kappa^{p}(\mathbb{R}^n))}\lesssim 2^{(-1)^{i+1}\epsilon_i k}\|\{\lambda_j\}^\infty_{j=1}\|_{\ell^{\alpha_i}}
    \end{equation*}
    for all $k\in\mathbb{Z}$, then we have
    \begin{equation*}
        \left\|\sum_{j=1}^\infty \lambda_j|g_j|^2\right\|_{L^{q,\infty}(\mathbb{R},L_\kappa^{p}(\mathbb{R}^n))}\lesssim \|\{\lambda_j\}^\infty_{j=1}\|_{\ell^{\alpha,1}}
    \end{equation*}
    for all sequences $\{\lambda_j\}^\infty_{j=1}\in \ell^{\alpha,1}$, where
    \begin{equation*}
        \frac{1}{q}=\frac{\theta}{q_0}+\frac{1-\theta}{q_1},\;\frac{1}{\alpha}=\frac{\theta}{\alpha_0}+\frac{1-\theta}{\alpha_1} \text{ and } \theta=\frac{\epsilon_1}{\epsilon_1+\epsilon_2}.
    \end{equation*}
\end{proposition}

\section{Technical Lemmas}\label{sec3}
In this section, we introduce some important lemmas to prove the main result of this paper.  Since the Dunkl transform of radial functions is explicitly expressed by the integral related to Bessel functions, we first list some properties of Bessel functions.\\

Let $J_\nu$ be the Bessel function of order $\nu>-\frac{1}{2}$, defined as
\begin{equation*}
J_\nu(r)=\frac{(\frac{r}{2})^\nu}{\Gamma(\nu+\frac{1}{2})\pi^{\frac{1}{2}}}\int_{-1}^1e^{ir\tau}(1-\tau^2)^{\nu-\frac{1}{2}}d\tau.
\end{equation*}

\begin{lemma}[see \cite{Grafakos2008}]\label{Bessel}
For $r>0$ and $\nu>-\frac{1}{2}$, we have
\begin{align}
&(1)J_\nu(r)\leq C_\nu r^\nu,0<r<1;\label{bessel1}\\
&(2)\frac{d}{dr}\left(r^{-\nu}J_\nu(r)\right)=-r^{-\nu}J_{\nu+1}(r);\label{bessel2}\\
&(3)J_\nu(r)\leq C_\nu r^{-\frac{1}{2}}, r\geq1;\label{bessel3}\\
&(4)J_\nu(r)=\frac{(\frac{r}{2})^\nu}{\Gamma(\nu+\frac{1}{2})\Gamma(\frac{1}{2})} \left[ie^{-ir}\int_0^\infty e^{-rt}(t^2+2it)^{\nu-\frac{1}{2}}dt-ie^{ir}\int_0^\infty e^{-rt}(t^2-2it)^{\nu-\frac{1}{2}}dt\right].\label{bessel4}
\end{align}
\end{lemma}
\begin{remark}
From \eqref{bessel1} and \eqref{bessel2} of Lemma \ref{Bessel}, we can easily obtain that for any $0\leq s\leq 1$ and $\beta\in\mathbb{N}$,
\begin{equation}\label{small-s}
  \left|\frac{d^\beta}{dr^\beta}\left(\psi^2(r)\frac{J_{\frac{N-2}{2}}(rs)}{(rs)^\frac{N-2}{2}}  r^{N-1}\right)\right|\leq C_{\beta},
\end{equation}
where $\psi$ is the function in the introdution.
Again from \eqref{bessel4}, we have the identity
\begin{equation}\label{Bessel-Fourier}
    \frac{J_{\frac{N-2}{2}}(r)}{r^\frac{N-2}{2}}=C\left(e^{ir}h(r)+e^{-ir}\overline{h(r)}\right),
\end{equation}
 where $$h(r)=-i\int_0^\infty e^{-rt}(t^2-2it)^\frac{N-3}{2}dt,$$
 and  for any $\beta\in\mathbb{N}$, one can get     
\begin{equation}\label{h}
\left|\frac{d^\beta}{dr^\beta}h(r)\right|\leq C_\beta(1+r)^{-\frac{N-1}{2}-\beta},
\end{equation}
for all $\beta\in\mathbb{N}$. Thus, from \eqref{h}, for any $s\geq1$ and $\beta\in\mathbb{N}$, we have
\begin{equation}\label{big-s}
\left|\frac{d^\beta}{dr^\beta}\left(\psi(r)^2h(rs) r^{N-1}\right)\right|\leq C_\beta s^{-\frac{N-1}{2}}.
\end{equation}
\end{remark}
We also exploit the following estimates, which can be easily proved by comparing the sums with the corresponding integrals.
\begin{lemma}\label{Sum}
Fix $\beta>0$. There exists $C_\beta>0$ such that for any $A>0$, we have
\begin{align*}
\sum_{j\in\mathbb{Z}, 2^j\leq A}2^{j\beta}&\leq C_\beta A^\beta,\\
\sum_{j\in\mathbb{Z}, 2^j>A}2^{-j\beta}&\leq C_\beta A^{-\beta}.
\end{align*}
\end{lemma}
The Hardy-Littlewood-Sobolev inequality plays a crucial role in the proof of Theorem  \ref{Strichartz-single}. To sharpen the inequalities in Theorem \ref{Strichartz-single} for functions in Lebesgue spaces, we will employ its refined form for functions in Lorentz spaces.
\begin{lemma}[see \cite{Neil}] \label{H-L-S-Lorentz}
For any $\sigma\in (0,1)$, $p_1,p_2,r_1,r_2\in (1,+\infty)$ such that 
\begin{equation*}
\frac{1}{p_1}+\frac{1}{p_2}+\sigma=2 \text{ and } \frac{1}{r_1}+\frac{1}{r_2}\geq1,
\end{equation*}
we have
\begin{equation*}
    \left|\int_\mathbb{R}\int_\mathbb{R}\frac{g(t)h(s)}{|t-s|^\sigma}dsdt\right|\lesssim \|g\|_{L^{p_1,r_1}(\mathbb{R})}\|h\|_{L^{p_2,r_2}(\mathbb{R})}.
\end{equation*}
\end{lemma}
We further need to prove the following new inequality,  which can be viewed as  a variant similar to Young's convolution inequality.
\begin{lemma}\label{new inequality} Let $1\leq p,q,r\leq \infty$ satisfy $1+\frac{1}{r}=\frac{1}{p}+\frac{1}{q}$. Suppose $f=\tilde{f}(|\cdot|)$ is a radial function in $L_\kappa^p(\mathbb{R}^n)$ and $g\in L_\kappa^q(\mathbb{R}^n)$. Then the function
\begin{equation*}
G(x)=\int_{\mathbb{R}^n}\tilde{f}\left(\big||x|-|y|\big|\right)g(y)h_\kappa(y) dy,
\end{equation*}
is a radial function in $L_\kappa^r(\mathbb{R}^n)$ satisfying
\begin{equation*}
\|G\|_{L_\kappa^r(\mathbb{R}^n)}\leq \|f\|_{L_\kappa^p(\mathbb{R}^n)} \|g\|_{L_\kappa^q(\mathbb{R}^n)}.  
\end{equation*}
\end{lemma}
\begin{proof} We assume $1\leq p,q,r<\infty$. From the expression of the function $h$, it is easy to see $h$ is radial and we denote $G(x)=\widetilde{G}(\rho)$, where $\rho=|x|$. Let $C_\kappa=\int_{\mathbb{S}^{n-1}}h_\kappa(y)d\sigma(y)$ and we have 
\begin{equation}\label{G}
\|G\|_{L_\kappa^r(\mathbb{R}^n)}=C_\kappa^\frac{1}{r}\|\widetilde{G}\|_{L^r(\mathbb{R}^+,\rho^{N-1}d\rho)}.
\end{equation}
We rewrite 
\begin{align*}
\widetilde{G}(\rho)&=\int_{\mathbb{R}^n}\tilde{f}\left(\big|\rho-|y|\big|\right) g(y)h_\kappa(y)dy\\
&=\int_{\mathbb{S}^{n-1}}\left(\int_0^\infty \tilde{f}\left(|\rho-s|\right) g(sy)s^{N-1}ds\right)h_\kappa(y)d\sigma(y).
\end{align*}
Using the generalized Minkowski's inequality, Young's convolution inequality and H\"older's inequality, we get
\begin{align*}
\|\widetilde{G}\|_{L^r(\mathbb{R}^+,\rho^{N-1}d\rho)}&\leq\int_{\mathbb{S}^{n-1}}\left\|\int_0^\infty \tilde{f}\left(|\cdot-s|\right) g(sy)s^{N-1}ds\right\|_{L^r(\mathbb{R}^+,\rho^{N-1}d\rho)}h_\kappa(y)d\sigma(y)\\
&\leq\int_{\mathbb{S}^{n-1}}\|\tilde{f}(|\cdot|)\|_{L^p(\mathbb{R}^+,\rho^{N-1}d\rho)}\|g(\cdot y)\|_{L^q(\mathbb{R}^+,\rho^{N-1}d\rho)}h_\kappa(y)d\sigma(y)\\ 
&=C_\kappa^{-\frac{1}{p}}\|f\|_{L_\kappa^p(\mathbb{R}^n)}\int_{\mathbb{S}^{n-1}}\|g(\cdot y)\|_{L^q(\mathbb{R}^+,\rho^{N-1}d\rho)}h_\kappa(y)d\sigma(y)\\ 
&\leq C_\kappa^{\frac{1}{q^\prime}-\frac{1}{p}}\|f\|_{L_\kappa^p(\mathbb{R}^n)}\left(\int_{\mathbb{S}^{n-1}}\|g(\cdot y)\|^q_{L^q(\mathbb{R}^+,\rho^{N-1}d\rho)}h_\kappa(y)d\sigma(y)\right)^\frac{1}{q}\\
&= C_\kappa^{-\frac{1}{r}}\|f\|_{L_\kappa^p(\mathbb{R}^n)}\|g\|_{L_\kappa^q(\mathbb{R}^n)},
\end{align*}
and the desired result follows immediately from \eqref{G}.

In the case where at least one of $p,q,r$ is $\infty$, by an argument analogous to the above, we can also derive the desired results.
\end{proof}
Our proof of frequency localized estimates is based on a bilinear interpolation argument inspired by ideas in \cite{KT}. First we recall some basic results about the $K$-method of real interpolation. We refer the reader to \cite{Bennett, Bergh-Lofstrom} for details on the development of this theory. Here we only recall the essentials to be used in the sequel. Let $A_0, A_1$ be Banach spaces contained in some larger Banach space $A$. For every $a\in A_0+A_1$ and $t>0$, we define the $K$-functional of real interpolation by
\begin{equation*}
K(t,a,A_0,A_1)=\inf_{a=a_0+a_1}\left(\|a_0\|_{A_0}+\|a_1\|_{A_1}\right).
\end{equation*}For $0<\theta<1$ and $1\leq q\leq \infty$, we denote by $(A_0,A_1)_{\theta,q}$ the real interpolation spaces between $A_0$ and $A_1$ defined as
\begin{equation*}
    (A_0,A_1)_{\theta,q}=\left\{a\in A_0+A_1: \|a\|_{(A_0,A_1)_{\theta,q}}=\left(\int_0^\infty \Big(t^{-\theta}K(t,a,A_0,A_1)\Big)^q\frac{dt}{t}\right)^\frac{1}{q}<\infty\right\}.
\end{equation*}
The bilinear interpolation we shall use is the following.
\begin{theorem}[see \cite{Bergh-Lofstrom}]\label{bilinear-argument} If $A_0,A_1,B_0,B_1,C_0,C_1$ are Banach spaces and the bilinear operator $T$ is bounded from 
\begin{align*}
    &A_0\times B_0\rightarrow C_0\\
    &A_0\times B_1\rightarrow C_1\\
    &A_1\times B_0\rightarrow C_1,
\end{align*}
then $T$ is also bounded from
\begin{equation*}
(A_0,A_1)_{\theta_0,pr}\times (B_0,B_1)_{\theta_1,qr}\rightarrow (C_0,C_1)_{\theta,r},
\end{equation*}
for any $0<\theta_0,\theta_1<\theta<1$, $1\leq p,q,r\leq\infty$ such that $\frac{1}{p}+\frac{1}{q}\geq 1$ and $\theta=\theta_0+\theta_1$.
\end{theorem}
The real interpolation space identities we shall use are
\begin{equation}\label{L-qp}
    \left(L^{q_0}(\mathbb{R},L_\kappa^{p_0}(\mathbb{R}^n)),L^{q_1}(\mathbb{R},L_\kappa^{p_1}(\mathbb{R}^n))\right)_{\theta,q}=L^{q}(\mathbb{R},L_\kappa^{p,q}(\mathbb{R}^n)),
\end{equation}
whenever $\theta\in (0,1)$, $p_0,p_1,q_0,q_1,p,q\in [1,\infty)$ are such $\frac{1}{q}=\frac{1-\theta}{q_0}+\frac{\theta}{q_1}$ and $\frac{1}{p}=\frac{1-\theta}{p_0}+\frac{\theta}{p_1}$, 
\begin{equation}\label{l-qp}
    \left(\ell^{q_0,r_0},\ell^{q_1,r_1}\right)_{\theta,r}=\ell^{q,r},
\end{equation}
whenever $\theta\in (0,1)$, $q_0,q_1,q\in [1,\infty)$ and $r_0,r_1,r\in[1,\infty]$ are such $\frac{1}{q}=\frac{1-\theta}{q_0}+\frac{\theta}{q_1}$, and
\begin{equation}\label{l-infty-s}
\left(\ell_\infty^{s_0},\ell_\infty^{s_1}\right)_{\theta,1}=\ell^s_1,
\end{equation}
whenever $\theta\in (0,1)$ and $s_0,s_1,s\in\mathbb{R}$ are such $s=(1-\theta)s_0+\theta s_1$, where $\ell_q^s=L^q(\mathbb{Z}, 2^{js}dj)$ are weighted sequence spaces and $dj$ is counting measure.

\section{Refinement of Theorem  \ref{Strichartz-single} and \ref{M-S} in Lorentz space}\label{sec4}
In this section, we generalize the estimates from Theorems \ref{Strichartz-single} and \ref{M-S} by upgrading them from Lebesgue spaces to Lorentz spaces. This refinement is a key step in the subsequent real interpolation process used to establish the global orthonormal results. We start with the following result. 
 \begin{theorem}
Assume $p,r\geq1$, $q>1$ and $\frac{2}{q}+\frac{N}{p}=N$. Then the Dunkl-Schr\"{o}dinger operator $e^{it\Delta_\kappa}$ is bounded from $L_\kappa^2(\mathbb{R}^n)$ to $L^{2q,2r}(\mathbb{R}, L_\kappa^{2p}(\mathbb{R}^n))$, i.e.,
\begin{equation*}
\|e^{it\Delta_\kappa}f\|_{L^{2q,2r}(\mathbb{R},L_\kappa^{2p}(\mathbb{R}^n))}\lesssim \|f\|_{L_\kappa^2(\mathbb{R}^n)}.
\end{equation*}    
\end{theorem}
\begin{proof}
    By duality, it is equivalent to prove that its adjoint operator $\left(e^{it\Delta_\kappa}\right)^*$ is bounded from $L^{(2q)^\prime,(2r)^\prime}(\mathbb{R}, L_\kappa^{(2p)^\prime}(\mathbb{R}^n))$ to $L_\kappa^2(\mathbb{R}^n)$. In fact, by \eqref{dispersive} and Lemma \ref{H-L-S-Lorentz}, we have
    \begin{align*}
       \left \|\left(e^{it\Delta_k}\right)^*g(x)\right\|^2_{L_\kappa^2(\mathbb{R}^n)}&=  \left \|\int_{\mathbb{R}}e^{-is\Delta_\kappa} g(x,s)ds\right\|^2_{L_\kappa^2(\mathbb{R}^n)}\\&=\int_{\mathbb{R}^2}\langle e^{-it\Delta_\kappa} g(\cdot,t),e^{-is\Delta_\kappa} g(\cdot,s)\rangle_{L_\kappa^2(\mathbb{R}^n)}dtds\\
        &=\int_{\mathbb{R}^2} \langle g(\cdot,t),e^{i(t-s)\Delta_\kappa} g(\cdot,s)\rangle_{L_\kappa^2(\mathbb{R}^n)}dtds\\
        &\leq \int_{\mathbb{R}^2} \|g(\cdot,t)\|_{L^{(2p)^\prime}(\mathbb{R}^n)}\|e^{i(t-s)\Delta_\kappa} g(\cdot,s)\|_{L^{2p}(\mathbb{R}^n)}dtds\\
        &\leq \int_{\mathbb{R}^2}\frac{\|g(\cdot,t)\|_{L^{(2p)^\prime}(\mathbb{R}^n)}\|g(\cdot,s)\|_{L^{(2p)^\prime}(\mathbb{R}^n)}}{|t-s|^\frac{N}{2p^\prime}} dtds\\
        &\leq \|g\|^2_{L^{(2q)^\prime,(2r)^\prime}(\mathbb{R}, L_\kappa^{(2p)^\prime}(\mathbb{R}^n))}.
    \end{align*} 
\end{proof}

\begin{theorem} \label{Lorentz-refined}
If $N\geq 1$, $(\frac{1}{p},\frac{1}{q})\in (A,F]$ and $\alpha=\frac{2p}{p+1}$, then we have
\begin{equation*}
			\bigg\|\sum_{j=1}^\infty \lambda_j|e^{it\Delta_\kappa}f_j|^2\bigg\|_{L^{q,\alpha}(\mathbb{R}, L_\kappa^p(\mathbb{R}^n))}\lesssim \|\{\lambda_j\}^\infty_{j=1}\|_{\ell^{\alpha}} 
\end{equation*}	
  holds for all families of orthonormal functions $\{f_j\}_{j=1}^\infty$ in $L_\kappa^2(\mathbb{R}^n)$ and all sequences $\{\lambda_j\}^\infty_{j=1}$ in $\ell^{\alpha}$. 
\end{theorem}
\begin{proof}Thanks to the duality principle Theorem \ref{duality-principle}, the desired estimate holds if and only if
\begin{equation}\label{T_1}
	\big\|WAA^*\overline{W}\big\|_{\mathfrak{S}^{\alpha^\prime}(L^2(\mathbb{R}, L_\kappa^2(\mathbb{R}^n)))}\lesssim \big\|W\big\|^2_{L^{2q^\prime,2\alpha^\prime}(\mathbb{R}, L_\kappa^{\alpha^\prime}(\mathbb{R}^n))}
		\end{equation}
  for all $W\in L^{2q^\prime,2\alpha^\prime}(\mathbb{R}, L_\kappa^{\alpha^\prime}(\mathbb{R}^n))$, where the operator $A$ is given by  $Af(x,t)=e^{it\Delta_\kappa}f(x)$.
 Now consider the analytic family of operators $T_z$ defined by
\begin{equation*}
    \mathcal{F}_{t,\kappa}(T_zg)(\xi,\tau)=\frac{1}{\Gamma(z+1)}(\tau-|\xi|^2)_+^z\mathcal{F}_{t,\kappa}g(\xi,\tau),
\end{equation*}
for $(\xi,\tau)\in\mathbb{R}^n\times\mathbb{R}$, where the Dunkl-Fourier transform of $g$ is defined as
\begin{equation*}
    \mathcal{F}_{t,\kappa}g(\xi,\tau)=\frac{1}{2\pi c_\kappa}\int_{\mathbb{R}}\int_{\mathbb{R}^{n}} g(x,t) e^{-it\tau}K(x,-i\xi)h_\kappa(x)dxdt.
\end{equation*}
Using the fact that $\mathcal{F}\left(\frac{t_+^z}{\Gamma(z+1)}\right)(\tau)=ie^{iz \pi/2}(\tau+i0)^{-z-1}=i\left(e^{iz \pi/2}\tau_+^{-z-1}-e^{-iz \pi/2}\tau^{-z-1}_-\right)$, where the latter equality is true for $z\not\in\mathbb{Z}$ (see \cite{Gelfand-Shilov}), we can conclude that
$T_{-1}=AA^*$.

 Therefore, by analytic interpolation, the required estimate \eqref{T_1}  will follow once we have the following two inequalities
\begin{equation}\label{plancherel}
    \big\|W_1T_z W_2\big\|_{\mathfrak{S}^\infty(L^2(\mathbb{R}, L_\kappa^2(\mathbb{R}^n)))}\leq C(Im(z))\big\|W_1\big\|_{L^\infty(\mathbb{R}, L_\kappa^\infty(\mathbb{R}^n))}\big\|W_2\big\|_{L^\infty(\mathbb{R}, L_\kappa^\infty(\mathbb{R}^n))}
\end{equation}
for $Re(z)=0$ and
\begin{equation}\label{T_z}
     \big\|W_1T_z W_2\big\|_{\mathfrak{S}^2(L^2(\mathbb{R}, L_\kappa^2(\mathbb{R}^n)))}\leq C(Im(z))\big\|W_1\big\|_{L^{\tilde{q},4}(\mathbb{R}, L_\kappa^2(\mathbb{R}^n))}\big\|W_2\big\|_{L^{\tilde{q},4}(\mathbb{R}, L_\kappa^2(\mathbb{R}^n))}
\end{equation}
for $Re(z)=-\frac{2}{\tilde{q}}-\frac{N}{2}$ and $\frac{N+1}{2}<-Re(z)\leq \frac{N+2}{2}$, where $C(Im(z))$ is a constant that grows exponentially with $Im(z)$.

When $Re(z)=0$, the estimate  \eqref{plancherel} follows immediately from Plancherel's theorem. On the other hand, if $Re(z)=-\frac{2}{\tilde{q}}-\frac{N}{2}$ and $\frac{N+1}{2}<-Re(z)\leq\frac{N+2}{2}$, the kernel of $T_z$ is given by
\begin{align*}
    &\quad \tau_x\left(\frac{1}{c_\kappa\Gamma(z+1)}\int_{\mathbb{R}}\tau_+^ze^{i\tau (t-s)}d\tau\int_{\mathbb{R}^n} e^{i(t-s)|\xi|^2} K(-y,i \xi)h_\kappa(\xi)d\xi\right)\\
    &=ie^{iz \pi/2}(t-s+i0)^{-z-1}\tau_x \mathcal{F}^{-1}_\kappa (e^{i(t-s)|\cdot|^2} )(-y)\\
    &=ie^{iz \pi/2}(t-s+i0)^{-z-1}\tau_x K_{-i(t-s)}(-y)
\end{align*}
and by \eqref{translator-decay}, we see that the kernel of $W_1T_z W_2$ is bounded by
$$C(Im(z))|W_1(x,t)||t-s|^{-Re(z)-1-\frac{N}{2}}|W_2(y,s)|.$$
Therefore, by using the  Hardy-Littlewood-Sobolev inequality in Lorentz spaces in Lemma \ref{H-L-S-Lorentz}, we have
\begin{align*}
   &\big\|W_1T_z W_2\big\|^2_{\mathfrak{S}^2(L^2(\mathbb{R}, L_\kappa^2(\mathbb{R}^n)))}\\
   &\leq C(Im(z))\int_{\mathbb{R}^2}\int_{\mathbb{R}^{2n}} |W_1(x,t)|^2|t-s|^{-2Re(z)-2-N}|W_2(y,s)|^2h_\kappa(x)h_\kappa(y)dxdydtds\\
   &= C(Im(z))\int_{\mathbb{R}^2}\frac{\|W_1(\cdot,t)\|^2_{L_\kappa^2(\mathbb{R}^n)}\|W_2(\cdot,s)\|^2_{L_\kappa^2(\mathbb{R}^n)}}{|t-s|^{2Re(z)+2+N}}dtds\\
   &\leq C(Im(z))\big\|W_1\big\|_{L^{\tilde{q},4}(\mathbb{R}, L_\kappa^2(\mathbb{R}^n))}\big\|W_2\big\|_{L^{\tilde{q},4}(\mathbb{R}, L_\kappa^2(\mathbb{R}^n))}
\end{align*}
and this completes the proof of the theorem.  
\end{proof}

\section{Proof of the frequency localized estimates in Theorem \ref{frequency-localized}}\label{sec5}
This section is devoted to prove the frequency localized estimates given in  Theorem \ref{frequency-localized}. Although the orthogonality is not preserved by the localised operator $P$, by the argument in Frank-Sabin \cite{FS}, we still obtain the following frequency localised estimates on the segment line $[B,A)$.
\begin{theorem} \label{adimissible}
If $N\geq 1$, $(\frac{1}{p},\frac{1}{q})\in [B,A)$ and $\alpha=\frac{2p}{p+1}$, then we have
\begin{equation*}
			\bigg\|\sum_{j=1}^\infty \lambda_j|e^{it\Delta_\kappa}Pf_j|^2\bigg\|_{L^q(\mathbb{R}, L_\kappa^p(\mathbb{R}^n))}\lesssim_\psi \|\{\lambda_j\}^\infty_{j=1}\|_{\ell^{\alpha}} 
\end{equation*}	
  holds for all families of orthonormal functions $\{f_j\}_{j=1}^\infty$ in $L_\kappa^2(\mathbb{R}^n)$ and all sequences $\{\lambda_j\}^\infty_{j=1}$ in $\ell^{\alpha}$. 
\end{theorem}
\begin{proof}
    From the fact that $e^{it\Delta_\kappa}$ is unitary on $L_\kappa^2(\mathbb{R}^n)$ and $\|Pf_j\|_{L_\kappa^2(\mathbb{R}^n)}\leq \|f_j\|_{L_\kappa^2(\mathbb{R}^n)}$, it immediately follows that
    \begin{equation*}
			\bigg\|\sum_{j=1}^\infty \lambda_j|e^{it\Delta_\kappa}Pf_j|^2\bigg\|_{L^\infty(\mathbb{R}, L_\kappa^1(\mathbb{R}^n))}\leq \sum_{j=1}^\infty |\lambda_j|\|e^{it\Delta_\kappa}Pf_j\|^2_{L^\infty(\mathbb{R}, L_\kappa^2(\mathbb{R}^n))}\leq \|\{\lambda_j\}^\infty_{j=1}\|_{\ell^1}. 
\end{equation*}	
By interpolation, it suffices to prove that 
\begin{equation*}
			\bigg\|\sum_{j=1}^\infty \lambda_j|e^{it\Delta_\kappa}Pf_j|^2\bigg\|_{L^q(\mathbb{R}, L_\kappa^p(\mathbb{R}^n))}\lesssim \|\{\lambda_j\}^\infty_{j=1}\|_{\ell^{\alpha}} 
\end{equation*}	
holds for $1+\frac{2}{N}\leq p<\frac{N+1}{N-1}$. By duality principle  given in Theorem \ref{duality-principle}, it is equivalent to  prove
\begin{equation}\label{P-T_1}
	\big\|We^{it\Delta_\kappa}(e^{it\Delta_\kappa})^*P^2\overline{W}\big\|_{\mathfrak{S}^{\alpha^\prime}(L^2(\mathbb{R}, L_\kappa^2(\mathbb{R}^n)))}\lesssim_\psi\big\|W\big\|^2_{L^{2q^\prime}(\mathbb{R}, L_\kappa^{\alpha^\prime}(\mathbb{R}^n))}
		\end{equation}
  for all $W\in L^{2q^\prime}(\mathbb{R}, L_\kappa^{\alpha^\prime}(\mathbb{R}^n))$, where the operator $e^{it\Delta_\kappa}(e^{it\Delta_\kappa})^*P^2$ is given by
  \begin{equation*}
      e^{it\Delta_\kappa}(e^{it\Delta_\kappa})^*P^2g(x)=\int_\mathbb{R} e^{i(t-s)\Delta_\kappa} P^2g(\cdot,s)(x)ds.
  \end{equation*}
 Similarly to Theorem \ref{Lorentz-refined},  consider the analytic family of operators $\mathcal{T}_z$ defined by
\begin{equation*}
    \mathcal{F}_{t,\kappa}(\mathcal{T}_zg)(\xi,\tau)=\frac{1}{\Gamma(z+1)}(\tau-|\xi|^2)_+^z\psi^2(|\xi|)\mathcal{F}_{t,\kappa}g(\xi,\tau)
\end{equation*}
for $(\xi,\tau)\in\mathbb{R}^n\times\mathbb{R}$. Then the kernel of $\mathcal{T}_z$ is given by
\begin{align*}
    &\quad \tau_x\left(\frac{1}{c_\kappa\Gamma(z+1)}\int_{\mathbb{R}}\tau_+^ze^{i\tau (t-s)}d\tau\int_{\mathbb{R}^n} e^{i(t-s)|\xi|^2} \psi^2(|\xi|)K(-y,i \xi)h_\kappa(\xi)d\xi\right)\\
    &=ie^{iz \pi/2}(t-s+i0)^{-z-1}\tau_x\mathcal{F}^{-1}_\kappa \left(e^{i(t-s)|\cdot|^2} \psi^2(|\cdot|)\right)(-y)
\end{align*}
and by \eqref{cut-off decay}, the kernel of $W_1\mathcal{T}_z W_2$ is also bounded by
$$C_\psi(Im(z))|W_1(x,t)||t-s|^{-Re(z)-1-\frac{N}{2}}|W_2(y,s)|.$$
Then, proceeding similarly as in  Theorem \ref{Lorentz-refined}, we have our required estimate \eqref{P-T_1}. 
\end{proof}

Now we are ready to prove Theorem \ref{frequency-localized}.
\begin{proof}[Proof of Theorem \ref{frequency-localized}]
In order to obtain our desired results, we begin with the elementary claim that
\begin{align}
			\bigg\|\sum_{j=1}^\infty \lambda_j|e^{it\Delta_\kappa}Pf_j|^2\bigg\|_{L^\infty(\mathbb{R}, L_\kappa^\infty(\mathbb{R}^n))}&\lesssim_\psi \|\{\lambda_j\}^\infty_{j=1}\|_{\ell^\infty},\label{infty}\\
   \bigg\|\sum_{j=1}^\infty \lambda_j|e^{it\Delta_\kappa}Pf_j|^2\bigg\|_{L^1(\mathbb{R}, L_\kappa^p(\mathbb{R}^n))}&\lesssim_\psi \|\{\lambda_j\}^\infty_{j=1}\|_{\ell^1},\text{ for any } \frac{N}{N-2}\leq p\leq\infty.\label{q=1}
\end{align}	
Indeed, for fixed  $(x,t)\in\mathbb{R}^n\times\mathbb{R}$, we define $\Psi_{x,t}(\xi)=K(ix,\xi)e^{-it|\xi|^2}\psi(|\xi|)$. Then from the orthonormality of $\{f_j\}_{j=1}^\infty$ in $L_\kappa^2(\mathbb{R}^n)$ and Bessel's inequality, it follows that 
\begin{align*}
  \sum_{j=1}^\infty |\lambda_j||e^{it\Delta_\kappa}Pf_j(x)|^2&=\frac{1}{c^2_\kappa} \sum_{j=1}^\infty |\lambda_j|\left|\int_{\mathbb{R}^n}e^{-it|\xi|^2}\psi(|\xi|)\mathcal{F}_\kappa f_j(\xi)K(ix,\xi)h_\kappa(\xi)d\xi\right|^2\\
  &=\frac{1}{c^2_\kappa} \sum_{j=1}^\infty |\lambda_j| \left|\int_{\mathbb{R}^n}\Psi_{x,t}(\xi) \mathcal{F}_\kappa f_j(\xi)h_\kappa(\xi)d\xi\right|^2\\
  &=\frac{1}{c^2_\kappa} \sum_{j=1}^\infty |\lambda_j| \left|\int_{\mathbb{R}^n}\mathcal{F}_\kappa \Psi_{x,t}(\xi)f_j(\xi)h_\kappa(\xi)d\xi\right|^2\\
  &\leq \frac{1}{c^2_\kappa} \|\mathcal{F}_\kappa \Psi_{x,t}\|^2_{L_\kappa^2(\mathbb{R}^n)}\|\{\lambda_j\}^\infty_{j=1}\|_{\ell^\infty} \\
  &=\frac{1}{c^2_\kappa} \|\Psi_{x,t}\|^2_{L_\kappa^2(\mathbb{R}^n)}\|\{\lambda_j\}^\infty_{j=1}\|_{\ell^\infty} \\
  &\leq \frac{1}{c^2_\kappa} \|\psi(|\cdot|)\|^2_{L_\kappa^2(\mathbb{R}^n)}\|\{\lambda_j\}^\infty_{j=1}\|_{\ell^\infty}\\
  &\lesssim_\psi \|\{\lambda_j\}^\infty_{j=1}\|_{\ell^\infty},
\end{align*}
which essentially the estimate \eqref{infty}. Next, we notice that, we can  rewrite $e^{it\Delta_\kappa}Pf_j$ as 
\begin{equation*}
    e^{it\Delta_\kappa}Pf_j=\mathcal{F}^{-1}_\kappa (\psi(|\cdot|))*_\kappa \left(e^{it\Delta_\kappa}f_j\right).
\end{equation*}
By Young's inequality \eqref{Young} and Theorem \ref{Strichartz-single},   for any $\frac{N}{N-2}\leq p\leq\infty$, we have
\begin{equation*}
\|e^{it\Delta_\kappa}Pf_j\|_{L^{2}(\mathbb{R},L_\kappa^{2p}(\mathbb{R}^n))}\leq \|\mathcal{F}^{-1}_\kappa (\psi(|\cdot|))\|_{L_\kappa^r(\mathbb{R}^n)} \|e^{it\Delta_\kappa}f_j\|_{L^{2}(\mathbb{R},L_\kappa^{\frac{2N}{N-2}}(\mathbb{R}^n))}\lesssim_\psi \|f_j\|_{L_\kappa^2(\mathbb{R}^n)},
\end{equation*}
where $1\leq r\leq\infty$ satisfies $1+\frac{1}{2p}=\frac{1}{r}+\frac{N-2}{2N}$. Now using the Minkowski's inequality, we get 
\begin{align*}
   \bigg\|\sum_{j=1}^\infty \lambda_j|e^{it\Delta_\kappa}Pf_j|^2\bigg\|_{L^1(\mathbb{R}, L_\kappa^p(\mathbb{R}^n))}&\leq    \sum_{j=1}^\infty |\lambda_j| \|e^{it\Delta_\kappa}Pf_j\|_{L^{2}(\mathbb{R},L_\kappa^{2p}(\mathbb{R}^n))}^2\lesssim_\psi \|\{\lambda_j\}^\infty_{j=1}\|_{\ell^\infty}. 
\end{align*}
and this completes the inequality \eqref{q=1}.

In light of Theorem \ref{adimissible}, \eqref{infty} and \eqref{q=1}, by complex interpolation, in order to prove Theorem \ref{frequency-localized}, it suffices to deduce that for any $(\frac{1}{p},\frac{1}{q})\in [E,A)$, i.e., $\frac{N+1}{N}<q\leq 2$  with $\frac{1}{q}=\frac{N}{(N-1)p}$, the inequality
\begin{equation}\label{goal}
			\bigg\|\sum_{j=1}^\infty \lambda_j|e^{it\Delta_\kappa}Pf_j|^2\bigg\|_{L^q(\mathbb{R}, L_\kappa^p(\mathbb{R}^n))}\lesssim_\psi \|\{\lambda_j\}^\infty_{j=1}\|_{\ell^q}, 
\end{equation}	
holds for all families of orthonormal functions $\{f_j\}_{j=1}^\infty$ in $L_\kappa^2(\mathbb{R}^n)$ and all sequences $\{\lambda_j\}^\infty_{j=1}$ in $\ell^q$. Using the duality principle Lemma \ref{duality-principle},  the inequality \eqref{goal} is equivalent to
\begin{equation}\label{new-goal}
	\big\|W_1e^{it\Delta_\kappa}(e^{it\Delta_\kappa})^*P^2W_2\big\|_{\mathfrak{S}^{q^\prime}(L^2(\mathbb{R}, L_\kappa^2(\mathbb{R}^n)))}\lesssim_\psi \big\|W_1\big\|_{L^{2q^\prime}(\mathbb{R}, L_\kappa^{2p^\prime}(\mathbb{R}^n))} \big\|W_2\big\|_{L^{2q^\prime}(\mathbb{R}, L_\kappa^{2p^\prime}(\mathbb{R}^n))},
		\end{equation}
  for all $W_1,W_2\in L^{2q^\prime}(\mathbb{R}, L_\kappa^{2p^\prime}(\mathbb{R}^n))$.
To establish this, first decompose the operator $e^{it\Delta_\kappa}(e^{it\Delta_\kappa})^*P^2$ dyadically as  $e^{it\Delta_\kappa}(e^{it\Delta_\kappa})^*P^2=\sum\limits_{j\in\mathbb{R}}T_j$ with
\begin{equation*}
    T_jg(x,t)=\int_\mathbb{R} \chi\left(2^{-j}(t-s)\right)e^{i(t-s)\Delta_\kappa}P^2g(\cdot,s)(x)ds,
\end{equation*}
where $\chi\in C^\infty_c(\mathbb{R})$ is a suitable function supported in $[\frac{1}{2},2]$ which satisfies $\sum\limits_{j\in\mathbb{Z}}\chi\left(2^{-j}t\right)=1$ for any $t\neq0$ and $\widehat{\chi}(0)=0$. It is easy to justify the existence of the function $\chi$ and 
\begin{equation}\label{chi}
  |\widehat{\chi}(\tau)|\lesssim\min\{|\tau|,|\tau|^{-1}\}.  
\end{equation}
Besides, we denote 
\begin{equation*}
\beta(r,s)=\frac{N+1}{2}-N\left(\frac{1}{r}+\frac{1}{s}\right).
\end{equation*}
In order to prove \eqref{new-goal}, we first estimate the Schatten bounds for a family of analytic operators $\sum\limits_{j\in\mathbb{Z}} W_12^{jz}T_j W_2$. Using the complex interpolation in Lemma \ref{interpolation schatten} between the two estimates in the following lemma, we  can obtain  our required estimate \eqref{new-goal}.               \end{proof}

\begin{lemma}For $z\in \mathbb{C}$, we have the following estimates
\begin{equation}\label{lemma-1}
	\big\|\sum\limits_{j\in\mathbb{Z}} 2^{jz}W_1T_jW_2\big\|_{\mathfrak{S}^\infty(L^2(\mathbb{R}, L_\kappa^2(\mathbb{R}^n)))}\lesssim_\psi \big\|W_1\big\|_{L^\infty(\mathbb{R}, L_\kappa^\infty(\mathbb{R}^n))} \big\|W_2\big\|_{L^\infty(\mathbb{R}, L_\kappa^\infty(\mathbb{R}^n))}
		\end{equation}
  for $Re(z)=-1$ and
    \begin{equation}\label{lemma-2}
	\big\|\sum\limits_{j\in\mathbb{Z}} 2^{jz}W_1T_jW_2\big\|_{\mathfrak{S}^{2}(L^2(\mathbb{R}, L_\kappa^2(\mathbb{R}^n)))}\lesssim_\psi \big\|W_1\big\|_{L^{4}(\mathbb{R}, L_\kappa^{r}(\mathbb{R}^n))} \big\|W_2\big\|_{L^{4}(\mathbb{R}, L_\kappa^{r}(\mathbb{R}^n))}
		\end{equation}
  for $Re(z)=-\beta(r,r)$ and $r\in (2,4)$.
\end{lemma}
\begin{proof}
First we consider the case $Re(z)=-1$.  Since
\begin{align*}
	&\big\|\sum\limits_{j\in\mathbb{Z}} 2^{jz}W_1T_jW_2\big\|_{\mathfrak{S}^\infty(L^2(\mathbb{R}, L_\kappa^2(\mathbb{R}^n)))}\\
 \leq &\big\|W_1\big\|_{L^\infty(\mathbb{R}, L_\kappa^\infty(\mathbb{R}^n))} \|\sum\limits_{j\in\mathbb{Z}} 2^{jz}T_j\|_{L^2(\mathbb{R}, L_\kappa^2(\mathbb{R}^n))\rightarrow L^2(\mathbb{R}, L_\kappa^2(\mathbb{R}^n))}\big\|W_2\big\|_{L^\infty(\mathbb{R}, L_\kappa^\infty(\mathbb{R}^n))},
		\end{align*}
it suffices to prove
\begin{equation}\label{2-2}
   \|\sum\limits_{j\in\mathbb{Z}} 2^{jz}T_jg\|_{L^2(\mathbb{R}, L_\kappa^2(\mathbb{R}^n))}\lesssim_\psi \|g\|_{L^2(\mathbb{R}, L_\kappa^2(\mathbb{R}^n))}.
\end{equation}
In fact, we have
\begin{equation*}
     \mathcal{F}_{t,\kappa}(T_jg)(\xi,\tau)=2^j\widehat{\chi}\left(2^j(\tau+|\xi|^2)\right)\psi^2(|\xi|)\mathcal{F}_{t,\kappa}g(\xi,\tau)
\end{equation*}
From Lemma \ref{Sum} and \eqref{chi}, it deduces that
\begin{align*}
    \sum\limits_{j\in\mathbb{Z}} \left|\widehat{\chi}\left(2^j(\tau+|\xi|^2)\right)\right|&=\left(\sum\limits_{2^j(\tau+|\xi|^2)>1} +\sum\limits_{2^j(\tau+|\xi|^2)\leq 1}\right)\left|\widehat{\chi}\left(2^j(\tau+|\xi|^2)\right)\right|\\
    &\lesssim \sum\limits_{2^j(\tau+|\xi|^2)>1} \left(2^j(\tau+|\xi|^2)\right)^{-1}+\sum\limits_{2^j(\tau+|\xi|^2)\leq 1} 2^j(\tau+|\xi|^2)\\
    &\lesssim 1
\end{align*}
holds uniformly in $(\xi,\tau)\in\mathbb{R}^n\times\mathbb{R}$ and \eqref{2-2} follows from Plancherel's formula when $Re(z)=-1$.

Next we turn to the case $Re(z)=-\beta(r,r)$. For this, we further break up each operator $T_j$ into two parts, i.e., $T_j=T^{(0)}_j+T^{(1)}_j$, where
\begin{align*}
    T^{(0)}_jg(x,t)&=\int_\mathbb{R} \chi\left(2^{-j}(t-s)\right)e^{i(t-s)\Delta_\kappa}P^2\left(E_{I(|x|,2^{j+4})^c}(|\cdot|) g(\cdot,s)\right)(x)ds,\\
    T^{(1)}_jg(x,t)&=\int_\mathbb{R} \chi\left(2^{-j}(t-s)\right)e^{i(t-s)\Delta_\kappa}P^2\left(E_{I(|x|,2^{j+4})}(|\cdot|) g(\cdot,s)\right)(x)ds,
\end{align*}
with  $I(|x|,2^{j+4})$ is the interval centered at $|x|$ with radius $2^{j+4}$ and $E_{I(|x|,2^{j+4})}$ is the characteristic function on $I(|x|,2^{j+4})$.

We begin with the part $T^{(0)}_j$, which also could be considered as an error term of $T_j$. We  can rewrite $T^{(0)}_jg(x,t)$ as 
\begin{equation*}
T^{(0)}_jg(x,t)=\int_\mathbb{R}\int_{\mathbb{R}^n}K^{(0)}_j(x,t,y,s) g(y,s)h_\kappa(y)dyds,
\end{equation*}
where the kernel of $W_1T^{(0)}_j W_2$ is given by
\begin{equation}\label{KernelWTW}
    K^{(0)}_j(x,t,y,s)=W_1(x,t)E_{I(0,2^{j+4})^c}(|x|-|y|)\chi\left(2^{-j}(t-s)\right)\tau_xe^{i(t-s)\Delta_\kappa} \Psi(-y)W_2(y,s)
\end{equation}
with $\mathcal{F}_\kappa \Psi(\xi)=\psi^2(|\xi|)$.

We claim that for any $t\in [2^{j-1},2^{j+1}]$, $x\in B(0,2^{j+4})^c$ and $M\in\mathbb{N},$ there exists a constant $C_{\psi,M}$ depending only on $\psi$ and $M$ such that
\begin{equation}\label{claim}
    \left|e^{it\Delta_\kappa} \Psi(x)\right|\leq  \frac{C_{\psi,M}}{(1+|x|)^M}.
\end{equation}
Indeed, from \eqref{radial-Dunkl} and \eqref{Bessel-Fourier}, we see that
\begin{align*}
    e^{it\Delta_\kappa} \Psi(x)&=\frac{1}{\Gamma(N/2)}\int_0^\infty e^{-itr^2}\psi^2(r)\frac{J_\frac{N-2}{2}(r|x|)}{\left(r|x|\right)^\frac{N-2}{2}}r^{N-1} dr\\
    &= C \int_0^\infty \left(e^{i(r|x|-tr^2)}h(r|x|)+e^{i(-r|x|-tr^2)}\overline{h(r|x|)}\right)\psi^2(r)r^{N-1} dr.
\end{align*}
From \eqref{small-s}, our claim holds true for any $|x|\leq 1$.  On the other hand, when $|x|>1$, it is enough to show the following oscillatory integral satisfies
\begin{equation}\label{claim-equiv}
\left|\int_0^\infty e^{i\left(\pm r|x|-tr^2\right)}\psi^2(r)H(r|x|)r^{N-1}dr\right|\leq  \frac{C_{\psi,M}}{|x|^M},
\end{equation}
where $H=h$ or $\overline{h}$. Note that $r\in[\frac{1}{2},2]$, $t\in [2^{j-1},2^{j+1}]$ and $x\in B(0,2^{j+4})^c$, we have 
\begin{equation*}
    \left|\frac{d}{dr}\left(\pm r|x|-tr^2\right)\right|=\left|\pm |x|-2tr\right|\geq |x|-4|t|\geq \frac{|x|}{2}
\end{equation*}
and hence, the estimate  \eqref{claim-equiv} follows immediately by repeated integration by parts and \eqref{big-s}.


Once we show \eqref{claim}, noting that $e^{it\Delta_\kappa} \Psi$ is radial, from \eqref{239} and \eqref{translator}, we have
\begin{equation*}
\left|\tau_xe^{i(t-s)\Delta_\kappa} \Psi(-y)\right|\leq \frac{C_{\psi,M}}{\left(1+\big||x|-|y|\big|\right)^M}.
\end{equation*}
Then from \eqref{KernelWTW}, we get 
\begin{align}
    |K_j^{(0)}(x,t,y,s)|&\leq C_{\psi,M}\frac{|W_1(x,t)||W_2(y,s)||\chi(2^{-j}(t-s))|}{\left(1+\big||x|-|y|\big|\right)^M}\nonumber\\
    &\leq C_{\psi,M}\frac{|W_1(x,t)||W_2(y,s)||\chi(2^{-j}(t-s))|}{(1+2^j)^\frac{M}{2}\left(1+\big||x|-|y|\big|\right)^\frac{M}{2}.\label{kernel-estimate}}
\end{align}
Therefore, for $Re(z)=-\beta(r,r), r\in (2,4)$ and $M\in\mathbb{N}$ sufficiently large, it follows from \eqref{kernel-estimate} and Lemma \ref{new inequality} that
\begin{align*}
    &\big\|\sum\limits_{j\in\mathbb{Z}} 2^{jz}W_1T^{(0)}_jW_2\big\|_{\mathfrak{S}^{2}(L^2(\mathbb{R}, L_\kappa^2(\mathbb{R}^n)))}\\
 &\leq \sum\limits_{j\in\mathbb{Z}} 2^{-j\beta(r,r)} \big\|W_1T_jW_2\big\|_{\mathfrak{S}^{2}(L^2(\mathbb{R}, L_\kappa^2(\mathbb{R}^n)))}\\
    &=\sum\limits_{j\in\mathbb{Z}} 2^{-j\beta(r,r)}\left(\int_{\mathbb{R}^2}\int_{\mathbb{R}^n}|K_j^{(0)}(x,t,y,s)|^2h_\kappa(x)h_\kappa(y)dxdydtds\right)^\frac{1}{2}\\
    &\leq C_{\psi,M}\sum\limits_{j\in\mathbb{Z}} \frac{2^{-j\beta(r,r)}}{(1+2^j)^M}\left(\int_{\mathbb{R}^2}\int_{\mathbb{R}^{2n}}\frac{|W_1(x,t)|^2|W_2(y,s)|^2|\chi(2^{-j}(t-s))|^2}{\left(1+\big||x|-|y|\big|\right)^{M}}h_\kappa(x)h_\kappa(y)dxdydtds\right)^\frac{1}{2}\\
    &\leq C_{\psi,M}\sum\limits_{j\in\mathbb{Z}} \frac{2^{-j\beta(r,r)}}{(1+2^j)^M}\left\|\frac{1}{(1+|\cdot|)^M}\right\|^\frac{1}{2}_{L_\kappa^\frac{r}{2(r-2)}(\mathbb{R}^n)}\left(\int_{\mathbb{R}^2} \|W_1(\cdot,t)\|^2_{L_\kappa^r(\mathbb{R}^n)}\|W_2(\cdot,s)\|^2_{L_\kappa^r(\mathbb{R}^n)}|\chi(2^{-j}(t-s))|^2dtds\right)^\frac{1}{2}\\
    &\leq C_{\psi,M}\sum\limits_{j\in\mathbb{Z}} \frac{2^{-j\beta(r,r)}}{(1+2^j)^M} \|\chi(2^{-j}\cdot)\|_{L^2(\mathbb{R})}\|W_1\|_{L^4(\mathbb{R},L^r_\kappa(\mathbb{R}^n))}\|W_2\|_{L^4(\mathbb{R},L^r_\kappa(\mathbb{R}^n))}\\
    &\leq C_{\psi,M}\sum\limits_{j\in\mathbb{Z}} \frac{2^{j(\frac{1}{2}-\beta(r,r))}}{(1+2^j)^M} \|W_1\|_{L^4(\mathbb{R},L^r_\kappa(\mathbb{R}^n))}\|W_2\|_{L^4(\mathbb{R},L^r_\kappa(\mathbb{R}^n))}.
\end{align*}

Note that for $r\in (2,4)$, $\frac{1}{2}-\beta(r,r)\in (0,\frac{N}{2})$ and hence we deduce the desired estimate
\begin{equation*}
    \big\|\sum\limits_{j\in\mathbb{Z}} 2^{jz}W_1T^{(0)}_jW_2\big\|_{\mathfrak{S}^{2}(L^2(\mathbb{R}, L_\kappa^2(\mathbb{R}^n)))}\lesssim_\psi \|W_1\|_{L^4(\mathbb{R},L^r_\kappa(\mathbb{R}^n))}\|W_2\|_{L^4(\mathbb{R},L^r_\kappa(\mathbb{R}^n))}.
\end{equation*}
As to the leading term $T_j^{(1)}$ of $T_j$, we rewrite 
\begin{equation*}
T^{(1)}_jg(x,t)=\int_\mathbb{R}\int_{\mathbb{R}^n}K^{(1)}_j(x,t,y,s) g(y,s)h_\kappa(y)dyds,
\end{equation*}
where the kernel of $W_1T^{(1)}_j W_2$ is given by
\begin{equation*}
    K^{(1)}_j(x,t,y,s)=W_1(x,t)E_{I(0,2^{j+4})}(|x|-|y|)\chi\left(2^{-j}(t-s)\right)\tau_x e^{i(t-s)\Delta_\kappa} \Psi(-y)W_2(y,s).
\end{equation*}
By using Cauchy-Schwarz inequality, it implies
\begin{align*}
    &\big\|\sum\limits_{j\in\mathbb{Z}} 2^{jz}W_1T^{(1)}_jW_2\big\|^2_{\mathfrak{S}^{2}(L^2(\mathbb{R}, L_\kappa^2(\mathbb{R}^n)))}\\
&=\int_{\mathbb{R}^2}\int_{\mathbb{R}^{2n}}\left|\sum\limits_{j\in\mathbb{Z}} 2^{jz}K^{(1)}_j(x,t,y,s)\right|^2h_\kappa(x) h_\kappa(y)dxdydtds\\
&=\int_{\mathbb{R}^2}\int_{\mathbb{R}^{2n}}\left|\sum\limits_{j\in\mathbb{Z}} 2^{jz}\chi\left(2^{-j}(t-s)\right)E_{I(0,2^{j+4})}(|x|-|y|)\right|^2\\
&\qquad\qquad\qquad\times \left|W_1(x,t)\tau_x e^{i(t-s)\Delta_\kappa} \Psi(-y)W_2(y,s)\right|^2h_\kappa(x) h_\kappa(y)dxdydtds\\
&\leq \int_{\mathbb{R}^2}\int_{\left||x|-|y|\right|\leq 2^{j+4}}\left(\sum\limits_{j\in\mathbb{Z}} 
|\chi\left(2^{-j}(t-s)\right)|\right)\sum\limits_{j\in\mathbb{Z}} 2^{-2\beta(r,r)j}|\chi\left(2^{-j}(t-s)\right)|\\
&\qquad\qquad\qquad\times \left|\tau_xe^{i(t-s)\Delta_\kappa} \Psi(-y)|^2|W_1(x,t)\right|^2|W_2(y,s)|^2h_\kappa(x) h_\kappa(y)dxdydtds\\
&\lesssim \sum\limits_{j\in\mathbb{Z}} 2^{-2\beta(r,r)j}\int_{\mathbb{R}^2}\int_{\left||x|-|y|\right|\leq 2^{j+4}}|\chi\left(2^{-j}(t-s)\right)|\left|\tau_x e^{i(t-s)\Delta_\kappa} \Psi(-y)\right|^2\\
&\qquad\qquad\qquad\times|W_1(x,t)|^2|W_2(y,s)|^2h_\kappa(x) h_\kappa(y)dxdydtds\\
&=\sum\limits_{j\in\mathbb{Z}} \left|T_{j,r}(|W_1|^2,|W_2|^2)\right|,
\end{align*}
where $T_{j,r}$ is the bilinear operator defined by
\begin{align*}
    T_{j,r}(V_1,V_2)&=2^{-2\beta(r,r)j}\int_{\mathbb{R}^2}\int_{\left||x|-|y|\right|\leq 2^{j+4}}|\chi\left(2^{-j}(t-s)\right)|\left|\tau_x e^{i(t-s)\Delta_\kappa} \Psi(-y)\right|^2\\
    &\qquad\qquad\qquad\times V_1(x,t) V_2(y,s)h_\kappa(x) h_\kappa(y)dxdydtds.
\end{align*}
In order to obtain
\begin{equation*}
 \big\|\sum\limits_{j\in\mathbb{Z}} 2^{jz}W_1T^{(0)}_jW_2\big\|^2_{\mathfrak{S}^{2}(L^2(\mathbb{R}, L_\kappa^2(\mathbb{R}^n)))} \lesssim_\psi \|W_1\|^2_{L^4(\mathbb{R},L^r_\kappa(\mathbb{R}^n))}\|W_2\|^2_{L^4(\mathbb{R},L^r_\kappa(\mathbb{R}^n))}
\end{equation*}
for each $r\in(2,4)$, it is equivalent to proving
\begin{equation}\label{bilinear}
\sum\limits_{j\in\mathbb{Z}} \left|T_{j,r}(V_1,V_2)\right|\lesssim_\psi \|V_1\|_{L^2(\mathbb{R},L^\frac{r}{2}_\kappa(\mathbb{R}^n))}\|V_2\|_{L^2(\mathbb{R},L^\frac{r}{2}_\kappa(\mathbb{R}^n))}.
\end{equation}
Our proof of \eqref{bilinear} is based on bilinear interpolation inspired by Keel-Tao \cite{KT} in the proof of the endpoint classical Strichartz estimates. For any $r\in (2,4)$ fixed, using the dispersive estimate for the Dunkl-Schr\"{o}dinger propagator \eqref{dispersive} and Lemma \ref{new inequality}, we have
\begin{align*}
    \left|T_{j,r}(V_1,V_2)\right|&\lesssim_\psi 2^{-2\beta(r,r)j}\int_{\mathbb{R}^2}\int_{\left||x|-|y|\right|\leq 2^{j+4}}|\chi\left(2^{-j}(t-s)\right)||t-s|^{-N}\\
    &\qquad\qquad\qquad\times |V_1(x,t)| |V_2(y,s)|h_\kappa(x) h_\kappa(y)dxdydtds\\
    &\lesssim_\psi 2^{-2\beta(r,r)j}2^{-Nj}\int_{\mathbb{R}^2}\int_{\left||x|-|y|\right|\leq 2^{j+4}}|\chi\left(2^{-j}(t-s)\right)|\\
    &\qquad\qquad\qquad\times |V_1(x,t)| |V_2(y,s)|h_\kappa(x) h_\kappa(y)dxdydtds\\
    &\lesssim_\psi 2^{-2\beta(r,r)j}2^{-Nj}\int_{\mathbb{R}^2}|\chi\left(2^{-j}(t-s)\right)|\\
    &\qquad\qquad\qquad\times \|I_{B(0,2^{j+4})(|\cdot|)}\|_{L^c_\kappa(\mathbb{R}^n)}\|V_1(\cdot,t)\|_{L^\frac{a}{2}_\kappa(\mathbb{R}^n)} \|V_1(\cdot,s)\|_{L^\frac{b}{2}_\kappa(\mathbb{R}^n)}dtds\\
    &\lesssim_\psi 2^{-2\beta(r,r)j}2^{N(1-\frac{2}{a}-\frac{2}{b})j}\int_{\mathbb{R}^2}|\chi\left(2^{-j}(t-s)\right)|\|V_1(\cdot,t)\|_{L^\frac{a}{2}_\kappa(\mathbb{R}^n)} \|V_1(\cdot,s)\|_{L^\frac{b}{2}_\kappa(\mathbb{R}^n)}dtds\\
\end{align*}
for any $a,b\geq 2$ and $c\geq 1$ such that $1\leq\frac{2}{a}+\frac{2}{b}=2-\frac{1}{c}$. Applying Young's convolution inequality, we have
\begin{align*}
    \left|T_{j,r}(V_1,V_2)\right|&\lesssim_\psi 2^{-2\beta(r,r)j}2^{N(1-\frac{2}{a}-\frac{2}{b})j} \|\chi(2^{-j}\cdot)\|_{L^1(\mathbb{R})}\|V_1\|_{L^2(\mathbb{R},L^\frac{a}{2}_\kappa(\mathbb{R}^n))}\|V_2\|_{L^2(\mathbb{R},L^\frac{b}{2}_\kappa(\mathbb{R}^n))}\\
    &\lesssim_\psi 2^{\alpha(a,b)j} \|V_1\|_{L^2(\mathbb{R},L^\frac{a}{2}_\kappa(\mathbb{R}^n))}\|V_2\|_{L^2(\mathbb{R},L^\frac{b}{2}_\kappa(\mathbb{R}^n))},
\end{align*}
where 
\begin{equation*}
    \alpha(a,b)=2\beta(a,b)-2\beta(r,r)=2N\left(\frac{2}{r}-\left(\frac{1}{a}+\frac{1}{b}\right)\right),
\end{equation*}
which indicates that the vector-valued bilinear operator $T=\{T_{j,r}\}_{j\in\mathbb{Z}}$ is bounded from
\begin{align*}
    &T: L^2(\mathbb{R},L^\frac{a}{2}_\kappa(\mathbb{R}^n))\times L^2(\mathbb{R},L^\frac{a}{2}_\kappa(\mathbb{R}^n))\rightarrow \ell_\infty^{-\alpha(a,a)}\\
    &T: L^2(\mathbb{R},L^\frac{a}{2}_\kappa(\mathbb{R}^n))\times L^2(\mathbb{R},L^\frac{b}{2}_\kappa(\mathbb{R}^n))\rightarrow \ell_\infty^{-\alpha(a,b)}\\
    &T: L^2(\mathbb{R},L^\frac{b}{2}_\kappa(\mathbb{R}^n))\times L^2(\mathbb{R},L^\frac{a}{2}_\kappa(\mathbb{R}^n))\rightarrow \ell_\infty^{-\alpha(b,a)}=\ell_\infty^{-\alpha(a,b)}
\end{align*}
with any $a,b\in (2,4)$. By the bilinear real interpolation argument in Theorem \ref{bilinear-argument}, it reveals that $T$ is also bounded from
\begin{small}\begin{equation*}
(L^2(\mathbb{R},L^\frac{a}{2}_\kappa(\mathbb{R}^n)),L^2(\mathbb{R},L^\frac{b}{2}_\kappa(\mathbb{R}^n)))_{\theta_0,p}\times (L^2(\mathbb{R},L^\frac{a}{2}_\kappa(\mathbb{R}^n)),L^2(\mathbb{R},L^\frac{b}{2}_\kappa(\mathbb{R}^n)))_{\theta_1,q}\rightarrow (\ell_\infty^{-\alpha(a,a)},\ell_\infty^{-\alpha(a,b)})_{\theta,1}
\end{equation*}
\end{small}
for any $0<\theta_0,\theta_1<\theta<1$, $1\leq p,q\leq\infty$ such that $\frac{1}{p}+\frac{1}{q}\geq 1$ and $\theta=\theta_0+\theta_1$.

Specifically, for each $r\in (2,4)$ fixed, we choose $\delta(r)>0$ sufficiently small such that $\frac{1}{a}=\frac{1}{r}+\delta(r)\in\left(\frac{1}{4},\frac{1}{2}\right)$ and $\frac{1}{b}=\frac{1}{r}-2\delta(r)\in\left(\frac{1}{4},\frac{1}{2}\right)$. 
Taking $p=q=2$, $\theta_0=\theta_1=\frac{1}{3}$, $\theta=\theta_0+\theta_1=\frac{2}{3}$, using the real interpolation identities in \eqref{L-qp} and \eqref{l-infty-s}, it follows that $T$ is bounded from
\begin{equation*}
    L^2(\mathbb{R},L^{\frac{r}{2},2}_\kappa(\mathbb{R}^n))\times L^2(\mathbb{R},L^{\frac{r}{2},2}_\kappa(\mathbb{R}^n))\rightarrow \ell^0_1.
\end{equation*}
By the embedding relation \eqref{embedding}, we have $L^2(\mathbb{R},L^{\frac{r}{2}}_\kappa(\mathbb{R}^n))\subseteq L^2(\mathbb{R},L^{\frac{r}{2},2}_\kappa(\mathbb{R}^n))$ for $r\in (2,4)$ and then we see $T$ is bounded from
\begin{equation*}
    L^2(\mathbb{R},L^{\frac{r}{2}}_\kappa(\mathbb{R}^n))\times L^2(\mathbb{R},L^{\frac{r}{2}}_\kappa(\mathbb{R}^n))\rightarrow \ell^0_1,
\end{equation*}
which essentially gives  \eqref{bilinear}.
\end{proof}

\begin{proof}[Proof of Remark \ref{N=2}, \ref{1<N<2} and \ref{N=1}]
From the proof of Theorem \ref{frequency-localized}, when $1<N\leq2$, the results \eqref{OAB} also hold true for $(\frac{1}{p},\frac{1}{q})$ belonging to $OAB\setminus A$ and $\alpha=\alpha^*(p,q)$. When $N=2$, the point $C$ and $D$ coincide and at this point the single-function Strichartz estimate \eqref{Single-Dunkl} fails. The boundedness of points belonging to int $OAC$ follows by the interpolation between points on the line segment $(A,C)$ arbitrarily close to the point $C$ and points belonging to int $OAB$ arbitrarily close to the line segment $(O,A)$. When $N=1$, the point $A$, $E$ and $G$ coincide and it follows from the interpolation between points on the line segment $(A,B]$ and the origin $O$.
\end{proof}
\section{Proof of the strong-type frequency global estimates in Theorem \ref{globally}}\label{sec6}
The main objective of this section is to establish the strong-type global frequency estimates stated in Theorem \ref{globally}. We will present the proof through a series of steps. First, we extend globally the following restricted weak-type estimates from the frequency localized estimates.
\begin{theorem}\label{restricted weak-type}
   Suppose $N\geq1$. If $(\frac{1}{p},\frac{1}{q})$ belongs to int $OAB$, $\alpha=\alpha^*(p,q)$, and $2s=N-\left(\frac{2}{q}+\frac{N}{p}\right)$, then
\begin{equation*}
	\bigg\|\sum_{j=1}^\infty \lambda_j|e^{it\Delta_\kappa}f_j|^2\bigg\|_{L^{q,\infty}(\mathbb{R}, L_\kappa^p(\mathbb{R}^n))}\lesssim \|\{\lambda_j\}^\infty_{j=1}\|_{\ell^{\alpha,1}} 
\end{equation*}	
  holds for all families of orthonormal functions $\{f_j\}_{j=1}^\infty$ in $\dot{H}_\kappa^s(\mathbb{R}^n)$ and all sequences $\{\lambda_j\}^\infty_{j=1}$ in $\ell^{\alpha,1}$. 
\end{theorem}
\begin{proof}
By Remark \ref{N=2}, \ref{1<N<2} and \ref{N=1}, for $N\geq 1$, we have  for $(\frac{1}{p},\frac{1}{q})$ belonging to $OAB\setminus [A,O)$ and $\alpha=\alpha^*(p,q)$,
\begin{equation*}
			\bigg\|\sum_{j=1}^\infty \lambda_j|e^{it\Delta_\kappa}P_0f_j|^2\bigg\|_{L^{q}(\mathbb{R}, L_\kappa^p(\mathbb{R}^n))}\lesssim \|\{\lambda_j\}^\infty_{j=1}\|_{\ell^{\alpha}} 
\end{equation*}	
 holds for all families of orthonormal functions $\{f_j\}_{j=1}^\infty$ in $L_\kappa^2(\mathbb{R}^n)$ and all sequences $\{\lambda_j\}^\infty_{j=1}$ in $\ell^{\alpha}$. By a rescaling argument, we have
 \begin{equation}\label{rescaling}
			\bigg\|\sum_{j=1}^\infty \lambda_j|e^{it\Delta_\kappa}P_k(-\Delta_\kappa)^{-\frac{s}{2}}f_j|^2\bigg\|_{L^{q}(\mathbb{R}, L_\kappa^p(\mathbb{R}^n))}\lesssim 2^{k(N-2s-\frac{2}{q}-\frac{N}{p})}\|\{\lambda_j\}^\infty_{j=1}\|_{\ell^{\alpha}}.
\end{equation}	
For any fixed point $(\frac{1}{p},\frac{1}{q})$ belonging to int $OAB$ and $2s=N-\left(\frac{2}{q}+\frac{N}{p}\right)$, select $\varepsilon>0$ sufficiently small such that $(\frac{1}{p_0},\frac{1}{q})$ and $(\frac{1}{p_1},\frac{1}{q})$ also belong to int $OAB$, where 
\begin{equation*}
    \frac{1}{q_i}=\frac{1}{q}+(-1)^i\frac{\varepsilon}{2}, \;i=0,1.
\end{equation*}
It follows from \eqref{rescaling} that for $i=0,1$
\begin{align*}
			\bigg\|\sum_{j=1}^\infty \lambda_j|e^{it\Delta_\kappa}P_k(-\Delta_\kappa)^{-\frac{s}{2}}f_j|^2\bigg\|_{L^{q_i,\infty}(\mathbb{R}, L_\kappa^p(\mathbb{R}^n))}&\leq \bigg\|\sum_{j=1}^\infty \lambda_j|e^{it\Delta_\kappa}P_k(-\Delta_\kappa)^\frac{s}{2}f_j|^2\bigg\|_{L^{q_i}(\mathbb{R}, L_\kappa^p(\mathbb{R}^n))}\\
   &\lesssim 2^{k(N-2s-\frac{2}{q_i}-\frac{N}{p})}\|\{\lambda_j\}^\infty_{j=1}\|_{\ell^{\alpha^*(p,q_i)}}\\
   &= 2^{(-1)^{i+1}\varepsilon_ik}\|\{\lambda_j\}^\infty_{j=1}\|_{\ell^{\alpha^*(p,q_i)}}.
\end{align*}	
 holds for all families of orthonormal functions $\{f_j\}_{j=1}^\infty$ in $L_\kappa^2(\mathbb{R}^n)$ and all sequences $\{\lambda_j\}^\infty_{j=1}$ in $\ell^{\alpha}$, where 
 $$\varepsilon_i=(-1)^{i+1}(N-2s-\frac{2}{q_i}-\frac{N}{p})=\varepsilon.$$
Applying Proposition \ref{upgrading}, it implies
\begin{equation*}
			\bigg\|\sum_{j=1}^\infty \lambda_j|e^{it\Delta_\kappa}(-\Delta_\kappa)^{-\frac{s}{2}}f_j|^2\bigg\|_{L^{q,\infty}(\mathbb{R}, L_\kappa^p(\mathbb{R}^n))}\lesssim \|\{\lambda_j\}^\infty_{j=1}\|_{\ell^{\alpha^*(p,q),1}}.
\end{equation*}	
holds for all families of orthonormal functions $\{f_j\}_{j=1}^\infty$ in $L_\kappa^2(\mathbb{R}^n)$ and all sequences $\{\lambda_j\}^\infty_{j=1}$ in $\ell^{\alpha,1}$, which is indeed our desired result.
\end{proof}

Next, based on the results in Theorem \ref{Lorentz-refined} and \ref{restricted weak-type}, by complex interpolation, we upgrade the restricted weak-type estimates to restricted strong-type estimates in interior of $OAF$.
\begin{theorem}\label{restricted strong-type}
   Suppose $N\geq1$. If $(\frac{1}{p},\frac{1}{q})$ belongs to int $OAF$, $\alpha=\alpha^*(p,q)$, and $2s=N-\left(\frac{2}{q}+\frac{N}{p}\right)$, then
\begin{equation*}
			\bigg\|\sum_{j=1}^\infty \lambda_j|e^{it\Delta_\kappa}f_j|^2\bigg\|_{L^{q}(\mathbb{R}, L_\kappa^p(\mathbb{R}^n))}\lesssim \|\{\lambda_j\}^\infty_{j=1}\|_{\ell^{\alpha,1}} 
\end{equation*}	
  holds for all families of orthonormal functions $\{f_j\}_{j=1}^\infty$ in $\dot{H}_\kappa^s(\mathbb{R}^n)$ and all sequences $\{\lambda_j\}^\infty_{j=1}$ in $\ell^{\alpha,1}$. 
\end{theorem}
\begin{proof}
    For any point $G=(\frac{1}{p},\frac{1}{q})$ belonging to int $OAF$ fixed, let $G_0=(\frac{1}{p_0},\frac{1}{q_0})$ be the intersection point of the line through the origin $O$ and $G$ with the line segment $(A,F)$ and we have $\frac{2}{q_0}+\frac{N}{q_0}=N$. Choose a point $G_1=(\frac{1}{p_1},\frac{1}{q_1})$ on the line segment $(O,G)$ which will be determined later. It is easy to see that $G_1$ belongs to int $OAB$ and $G$ is on the line segment $(G_0,G_1)$. Then there exists $\theta\in(0,1)$ such that $G=(1-\theta)G_0+\theta G_1$, i.e.,
    \begin{equation}\label{interpolation-index}
        \frac{1}{p}=\frac{1-\theta}{p_0}+\frac{\theta}{p_1} \text{ and }\frac{1}{q}=\frac{1-\theta}{q_0}+\frac{\theta}{q_1}. 
    \end{equation}
By the application of Theorem \ref{Lorentz-refined} and \ref{restricted weak-type}, we have 
\begin{equation*}
 \bigg\|\sum_{j=1}^\infty \lambda_j|e^{it\Delta_\kappa}f_j|^2\bigg\|_{L^{q_0,\alpha^*(p_0,q_0)}(\mathbb{R}, L_\kappa^{p_0}(\mathbb{R}^n))}\lesssim \|\{\lambda_j\}^\infty_{j=1}\|_{\ell^{\alpha^*(p_0,q_0)}},
\end{equation*}
 holds for all families of orthonormal functions $\{f_j\}_{j=1}^\infty$ in $L_\kappa^2(\mathbb{R}^n)$ and all sequences $\{\lambda_j\}^\infty_{j=1}$ in $\ell^{\alpha^*(p_0,q_0)}$,
 and for $2s_1=N-\left(\frac{2}{q_1}+\frac{N}{p_1}\right)$
\begin{equation*}
 \bigg\|\sum_{j=1}^\infty \lambda_j|e^{it\Delta_\kappa}f_j|^2\bigg\|_{L^{q_1,\infty}(\mathbb{R}, L_\kappa^{p_1}(\mathbb{R}^n))}\lesssim \|\{\lambda_j\}^\infty_{j=1}\|_{\ell^{\alpha^*(p_1,q_1),1}},  
\end{equation*}
holds for all families of orthonormal functions $\{f_j\}_{j=1}^\infty$ in $\dot{H}_\kappa^{s_1}(\mathbb{R}^n)$ and all sequences $\{\lambda_j\}^\infty_{j=1}$ in $\ell^{\alpha^*(p_1,q_1),1}$.

Using complex interpolation between the above two estimates with $\theta$,   for $\frac{1}{r}=\frac{1-\theta}{\alpha^*(p_0,q_0)}+\frac{\theta}{\infty}$, $\frac{1}{\tau}=\frac{1-\theta}{\alpha^*(p_0,q_0)}+\frac{\theta}{1}$, $\frac{1}{\alpha^*(p,q)}=\frac{1-\theta}{\alpha^*(p_0,q_0)}+\frac{\theta}{\alpha^*(p_1,q_1)}$ and $2s=2\left((1-\theta)\cdot 0+\theta s_1\right)= N-\left(\frac{2}{q}+\frac{N}{p}\right)$, we get 
\begin{equation*}
 \bigg\|\sum_{j=1}^\infty \lambda_j|e^{it\Delta_\kappa}f_j|^2\bigg\|_{L^{q,r}(\mathbb{R}, L_\kappa^{p}(\mathbb{R}^n))}\lesssim \|\{\lambda_j\}^\infty_{j=1}\|_{\ell^{\alpha^*(p,q),\tau}}
\end{equation*}
  for all families of orthonormal functions $\{f_j\}_{j=1}^\infty$ in $\dot{H}_\kappa^{s}(\mathbb{R}^n)$ and all sequences $\{\lambda_j\}^\infty_{j=1}$ in $\ell^{\alpha^*(p_1,q_1),\tau}$.
Since $\tau\geq 1$,  from the embedding relation of Lorentz spaces \eqref{embedding}, we deduce that
\begin{equation*}
 \bigg\|\sum_{j=1}^\infty \lambda_j|e^{it\Delta_\kappa}f_j|^2\bigg\|_{L^{q,r}(\mathbb{R}, L_\kappa^{p}(\mathbb{R}^n))}\lesssim \|\{\lambda_j\}^\infty_{j=1}\|_{\ell^{\alpha^*(p,q),1}}. 
\end{equation*}
We observe that $G_0=(\frac{1}{p_0},\frac{1}{q_0})$ is on the line segment (A,F) satisfying $1\leq \alpha^*(p_0,q_0)<q_0<q$ and $1\leq q\leq q_1<\infty$. In order to obtain our desired results, we should choose proper $G_1$ on the line segment $(O,G)$ such that $r=q$. Indeed, once we select the interpolation index $\theta$, we will determine the point $G_1$. We choose $\theta=1-\frac{\alpha^*(p_0,q_0)}{q}\in (0,1)$ and one can check that the point $G_1=(\frac{1}{p_1},\frac{1}{q_1})$ defined by \eqref{interpolation-index} is on the line segment $(O,G)$.
\end{proof}

Finally, we upgrade the restricted strong-type estimates in Theorem \ref{restricted strong-type} to the strong-type estimates by a succession of interpolation arguments.

\begin{proof}[Proof of Theorem \ref{globally}]
Firstly, we claim that it suffices to prove the desired results for all $(\frac{1}{p},\frac{1}{q})$ belonging to int $OAF$. Once we obtain these, by the interpolation with the estimates in Theorem \ref{M-S} on the line segment $(A,B]$, we can prove the desired estimates for all $(\frac{1}{p},\frac{1}{q})$ belonging to int $OAB$. It will also quickly imply the results in Theorem \ref{M-S} $(2)$. Indeed, for $N=2$, the point $C$ and $D$ coincide and at this point the single-function Strichartz estimate \eqref{Single-Dunkl} fails. The boundedness of points belonging to int $OAC$ follows by the interpolation between points on the line segment $(A,C)$ arbitrarily close to the point $C$ and points belonging to int $OAF$ arbitrarily close to the line segment $(O,A)$. When $N>2$, the single-funtion Strichartz estimate \eqref{Sobolev-single} and the triangle inequality immediately indicate that the inequality \eqref{OAB} holds for points $(\frac{1}{p},\frac{1}{q})$ belonging to the line segment $[C,D)$ and $\alpha=1$. Interpolating with points belonging to int $OAF$ arbitrarily close to the line segment $(O,A)$, we get the boundedness of points belonging to int $OACD$.

For any point $(\frac{1}{p},\frac{1}{q})$ belonging to int $OAF$ fixed and $2s=N-\left(\frac{2}{q}+\frac{N}{p}\right)$, we choose $\delta>0$ sufficiently small such that for $(\frac{1}{p_0},\frac{1}{q_0})$ and $(\frac{1}{p_1},\frac{1}{q_1})$ also belong to int $OAF$, where 
\begin{equation*}
    \frac{1}{p_i}=\frac{1}{p}+(-1)^i\frac{\delta}{N} \quad\text{and}\quad \frac{1}{q_i}=\frac{1}{q}+(-1)^{i+1}\frac{\delta}{2},\quad i=0,1.
\end{equation*}
It is clear that $2s=N-\left(\frac{2}{q_i}+\frac{N}{p_i}\right)$ for $i=0,1$, $\frac{1}{p}=\frac{1}{2}\left(\frac{1}{p_0}+\frac{1}{p_1}\right)$, $\frac{1}{q}=\frac{1}{2}\left(\frac{1}{q_0}+\frac{1}{q_1}\right)$ and $\frac{1}{\alpha^*(p,q)}=\frac{1}{2}\left(\frac{1}{\alpha^*(p_0,q_0)}+\frac{1}{\alpha^*(p_1,q_1)}\right)$. It follows from Theorem \ref{restricted strong-type} that for $i=0,1$ and all families of orthonormal functions $\{f_j\}_{j=1}^\infty$ in the common space $\dot{H}_\kappa^{s}(\mathbb{R}^n)$ with
\begin{equation}\label{i=0,1}
			\bigg\|\sum_{j=1}^\infty \lambda_j|e^{it\Delta_\kappa}f_j|^2\bigg\|_{L^{q_i}(\mathbb{R}, L_\kappa^{p_i}(\mathbb{R}^n))}\lesssim \|\{\lambda_j\}^\infty_{j=1}\|_{\ell^{\alpha^*(p_i,q_i),1}}, 
\end{equation}	
    for all sequences $\{\lambda_j\}^\infty_{j=1}$ in $\ell^{\alpha^*(p_i,q_i),1}$.  Moreover, from   \eqref{L-qp} and \eqref{l-qp}, it yields 
  \begin{align*}
       \left(L^{q_0}(\mathbb{R},L_\kappa^{p_0}(\mathbb{R}^n)),L^{q_1}(\mathbb{R},L_\kappa^{p_1}(\mathbb{R}^n))\right)_{\frac{1}{2},q}&=L^{q}(\mathbb{R},L_\kappa^{p,q}(\mathbb{R}^n)),\\
        \left(\ell^{\alpha^*(p_0,q_0),1},\ell^{\alpha^*(p_1,q_1),1}\right)_{\frac{1}{2},q}&=\ell^{\alpha^*(p,q),q}.
  \end{align*}
Using real interpolation on the estimates \eqref{i=0,1}, we obtain
\begin{equation*}
			\bigg\|\sum_{j=1}^\infty \lambda_j|e^{it\Delta_\kappa}f_j|^2\bigg\|_{L^{q}(\mathbb{R}, L_\kappa^{p,q}(\mathbb{R}^n))}\lesssim \|\{\lambda_j\}^\infty_{j=1}\|_{\ell^{\alpha^*(p,q),q}}.
\end{equation*}	
Note that for any $(\frac{1}{p},\frac{1}{q})$ belonging to int $OAF$ it satisfies $\alpha^*(p,q)<q<p$ and by the embedding relation of Lorentz spaces \eqref{embedding}, we have
\begin{equation*}
			\bigg\|\sum_{j=1}^\infty \lambda_j|e^{it\Delta_\kappa}f_j|^2\bigg\|_{L^{q}(\mathbb{R}, L_\kappa^{p}(\mathbb{R}^n))}\lesssim \|\{\lambda_j\}^\infty_{j=1}\|_{\ell^{\alpha^*(p,q)}},
\end{equation*}	
which are our desired estimates.
\end{proof}

\section*{Acknowledgments} SSM is supported by the DST-INSPIRE Faculty Fellowship DST/INSPIRE/04/2023/002038. MS is supported by the Guangdong Basic and Applied Basic Research Foundation (Grant No. 2023A1515010656). HW is supported by the National Natural Science Foundation of China (Grant No. 12171399 and 12271041).

\end{document}